\documentclass[12pt]{amsart}

\usepackage{amscd,amssymb,amsmath,amsfonts,amsthm,mathrsfs}
\usepackage{fullpage}

\usepackage{color}
\usepackage[backrefs, alphabetic]{amsrefs}
\usepackage[all]{xy} 

\DeclareMathOperator{\Hom}{Hom}
\DeclareMathOperator{\Gal}{Gal}
\DeclareMathOperator{\imm}{im}

\DeclareMathOperator{\res}{res}

\DeclareMathOperator{\Char}{char}

\DeclareMathOperator{\Fun}{Fun}
\let\H\relax
\DeclareMathOperator{\H}{H}

\newtheorem{theorem}{Theorem}

\newtheorem{proposition}{Proposition}[section]
\newtheorem{lemma}[proposition]{Lemma}

\newtheorem{corollary}[proposition]{Corollary}

\theoremstyle{definition}

\newtheorem{remark}[proposition]{Remark}

\numberwithin{equation}{section}

\newcommand{\hotimes}{\widehat\otimes}

\newcommand{\Z}{\mathbb{Z}}

\newcommand{\Q}{\mathbb{Q}}

\newcommand{\Kbb}{\mathbb{K}}

\newcommand{\Gc}{\mathcal{G}}

\newcommand{\Mcal}{\mathcal{M}}
\newcommand{\Hcal}{\mathcal{H}}

\newcommand{\Kfr}{\mathfrak{K}}
\newcommand{\Sfr}{\mathfrak{S}}

\newcommand{\Ffrak}{\mathfrak{F}}
\newcommand{\gfrak}{\mathfrak{g}}
\newcommand{\Gfrak}{\mathfrak{G}}
\newcommand{\Hfrak}{\mathfrak{H}}

\newcommand{\Lbb}{\mathbb{L}}

\newcommand{\tsigma}{\tilde\sigma}
\newcommand{\ttau}{\tilde\tau}

\renewcommand{\inf}{\operatorname{inf}}

\begin{document}
\title{Abelian-by-Central Galois Groups of Fields I:\\ a Formal Description}
\author{Adam Topaz}
\thanks{This research was supported by NSF postdoctoral fellowship DMS-1304114.}
\address{Adam Topaz \vskip 0pt
Department of Mathematics \vskip 0pt
University of California, Berkeley \vskip 0pt
970 Evans Hall \#3840 \vskip 0pt
Berkeley, CA. 94720-3840 \vskip 0pt
USA}
\email{atopaz@math.berkeley.edu}
\urladdr{http://math.berkeley.edu/~atopaz}
\date{\today}

\begin{abstract}
Let $K$ be a field whose characteristic is prime to a fixed integer $n$ with $\mu_n \subset K$, and choose $\omega \in \mu_n$ a primitive $n$th root of unity.
Denote the absolute Galois group of $K$ by $\Gal(K)$, and the mod-$n$ central-descending series of $\Gal(K)$ by $\Gal(K)^{(i)}$.
Recall that Kummer theory, together with our choice of $\omega$, provides a functorial isomorphism between $\Gal(K)/\Gal(K)^{(2)}$ and $\Hom(K^\times,\Z/n)$.
Analogously to Kummer theory, in this note we use the Merkurjev-Suslin theorem to construct a continuous, functorial and explicit embedding $\Gal(K)^{(2)}/\Gal(K)^{(3)} \hookrightarrow \Fun(K\smallsetminus\{0,1\},(\Z/n)^2)$, where $\Fun(K\smallsetminus\{0,1\},(\Z/n)^2)$ denotes the group of $(\Z/n)^2$-valued \emph{functions} on $K\smallsetminus\{0,1\}$.
We explicitly determine the functions associated to the image of commutators and $n$th powers of elements of $\Gal(K)$ under this embedding.
We then apply this theory to prove some new results concerning relations between elements in abelian-by-central Galois groups.
\end{abstract}

\maketitle

\section{Introduction}

In recent years, it has become increasingly evident that much of the arithmetic and geometric information which is encoded in very large Galois groups (e.g. maximal pro-$p$ Galois groups and absolute Galois groups) is already encoded in much smaller quotients and specifically in so-called ``abelian-by-central'' quotients.
Recall that the mod-$n$ abelian-by-central Galois group of a field $K$ is the maximal Galois group of $K$ which is a central extension of an exponent-$n$ abelian group by an exponent-$n$ abelian group.
In other words, the mod-$n$ abelian-by-central Galois group of $K$ is the quotient of its absolute Galois group associated to the 3rd term in the mod-$n$ central descending series.

One instance of this phenomenon comes from valuation theory.
It has been known for several years that one can detect valuations of a given field using ``large'' Galois groups.
More precisely, it was shown by {\sc Engler-Nogueira} \cite{Engler1994} and {\sc Efrat} \cite{Efrat1995} for $p=2$, and by {\sc Engler-Koenigsmann} \cite{Engler1998} for $p > 2$, that one can detect inertia/decomposition groups of tamely-branching valuations in the maximal pro-$p$ Galois group of a field $K$ of characteristic different from $p$ with $\mu_p \subset K$. 
Also in the same direction, {\sc Koenigsmann} \cite{Koenigsmann2003} shows how one can detect inertia/decomposition groups of valuations in absolute Galois groups of arbitrary fields. 
On the other hand, similar valuation-theoretic data is already encoded in abelian-by-central Galois groups.
It was shown by {\sc Bogomolov-Tschinkel} \cite{Bogomolov2007} and {\sc Pop} \cite{Pop2010} that one can detect \emph{abelian} inertia/decomposition groups of valuations using the pro-$p$ abelian-by-central Galois group of a function field over an \emph{algebraically closed field}.
More recently, {\sc Efrat-Min\'a\v{c}} \cite{Efrat2011a} proved that the mod-$p$ abelian-by-central Galois group encodes the existence (or non-existence) of a tamely-branching $p$-\emph{Henselian} valuation in a field $K$ of characteristic $\neq p$ with $\mu_p \subset K$.
Finally, it was shown by the author in {\sc Topaz} \cite{Topaz2012c} that the mod-$p^n$ abelian-by-central Galois group (for arbitrary $n$) of a field $K$ encodes abelian inertia/decomposition groups of almost \emph{arbitrary} valuations, as long as $\Char K \neq p$ and $K$ contains sufficiently many roots of unity.

Further instances of this phenomenon arise from Galois cohomology.
With mod-$2$ Galois cohomology, it was shown by {\sc Min\'a\v{c}-Spira} \cite{Minac1996} that the mod-$2$ abelian-by-central Galois group of a field $K$ can be seen as a ``Galois-theoretical analogue'' of the Witt ring of quadratic forms of $K$.
More recently, {\sc Chebolu-Efrat-Min\'a\v{c}} \cite{Chebolu2009} proved that one can recover the mod-$p^n$ Galois cohomology of a field $K$, endowed with the cup product and Bockstein morphism, as the \emph{decomposable} part of the mod-$p^n$ cohomology of the mod-$p^n$ abelian-by-central Galois group, as long as $K$ has characteristic different from $p$ with $\mu_{p^n} \subset K$; loc.cit. also gives a \emph{non-functorial} construction which determines the mod-$p^n$ abelian-by-central Galois group from the lower mod-$p^n$ Galois cohomology groups of the field.
See also {\sc Efrat-Min\'a\v{c}} \cite{Efrat2011c} for related results in this direction.
Along similar lines, the abelian-by-central Galois group, and its natural \emph{meta-abelian} generalizations, encode a lot of information about important Galois modules which naturally arise from \emph{strictly larger} quotients of the absolute Galois group via Kummer theory and Galois cohomology, as was shown by {\sc Min\'a\v{c}-Swallow-Topaz} \cite{Topaz2012a}.

The strongest instance of this phenomenon, however, is a conjecture in birational anabelian geometry, which was first proposed by {\sc Bogomolov} \cite{Bogomolov1991} (see also {\sc Pop} \cite{Pop2011} for a precise functorial formulation), whose goal is to recover a function field $K$ of dimension $\geq 2$ over an algebraically closed field from its pro-$p$ abelian-by-central Galois group. 
If successful, this conjecture would go far beyond Grothendieck's original (birational) anabelian conjectures (see \cite{Grothendieck}) since it deals with fields of \emph{purely geometric} nature, and with abelian-by-central Galois groups, which are ``almost-abelian,'' as opposed to absolute Galois groups.
While Bogomolov's conjecture is open in general, it has been proven for function fields of transcendence degree $\geq 2$ over the \emph{algebraic closure of a finite field} by {\sc Bogomolov-Tschinkel} \cite{Bogomolov2008a}, \cite{Bogomolov2011} and by {\sc Pop} \cite{Pop2003}, \cite{Pop2010}, \cite{Pop2011}.
A weaker version of this conjecture, which uses large Galois groups instead of abelian-by-central ones, was also resolved for function fields over $\bar \Q$, in transcendence degree $2$ by {\sc Silberstein} \cite{Silberstein2012} and in transcendence degree $\geq 3$ by {\sc Pop} \cite{Pop2011b}.
It is also important to note that the abelian-by-central results in valuation theory mentioned above (\cite{Bogomolov2007} and \cite{Pop2010} in particular) play a central role in the proof of the known cases of this conjecture.

In this note, we begin an investigation of mod-$n$ abelian-by-central Galois group in general.
The main result of this note uses the Merkurjev-Suslin theorem \cite{Merkurjev1982} to provide a \emph{formal} description of mod-$n$ abelian-by-central Galois groups, and can be seen as an extension of Kummer Theory to this non-abelian situation.
This gives a \emph{direct} and very \emph{explicit} link between the arithmetic structure of a field $(K,+,\cdot)$, and the structure of its mod-$n$ abelian-by-central Galois group, without doing arithmetic in any proper extensions of $K$.
More precisely, just as Kummer Theory describes the maximal mod-$n$ \emph{abelian} Galois group of a field $K$ (which contains sufficiently many roots of unity) as the group of homomorphisms $K^\times \rightarrow \Z/n$, in this note we give a description of the ``central'' part (i.e. the \emph{non-abelian} part) of the mod-$n$ abelian-by-central Galois group in terms of special \emph{functions} $K\smallsetminus\{0,1\} \rightarrow (\Z/n)^2$.
This new description is functorial in $K$ and is compatible with Kummer theory via taking commutators of pairs of elements and raising elements to $n$th powers.
See Theorem \ref{thm:main-thm-ffrak} below for the precise statement.

We use this formal description of the mod-$n$ abelian-by-central Galois group to determine the existence of certain types of relations in such groups only by considering \emph{Heisenberg Quotients} -- that is, homomorphisms from the mod-$n$ abelian-by-central Galois group of a field to the Heisenberg Group over $\Z/n$.
Similar results in this direction were shown by {\sc Efrat-Min\'a\v{c}} \cite{Efrat2011c}, \cite{Efrat2011b} in the mod-$p$ case.
See Corollary \ref{cor:cor-from-intro}, \S\ref{sec:an-exampl-heis}, \S\ref{sec:applications} and Theorem \ref{thm:application-to-relations} for more details.

Finally, using the results involving homomorphisms to the Heisenberg group, we conclude this note by observing that the notion of a \emph{commuting-liftable} pair from \cite{Topaz2012c} is equivalent to an \emph{a priori} more general notion, which we call \emph{weakly-commuting-liftable}, whenever one deals with mod-$n$ abelian-by-central Galois groups of a field.
This thereby generalizes the main results of loc.cit. which detect valuations using mod-$p^n$ abelian-by-central Galois groups, to even more minimal situations.
In addition to this observation, the main results of this note can be seen as a direct generalization of \S7 of loc.cit.
See \S\ref{sec:aappl-theor-comm} and Remark \ref{remark:cl-vs-wcl} of this paper for more details concerning commuting-liftable vs. weakly-commuting-liftable pairs.

\subsection{Notation and the Main Theorem}
\label{sec:notat-main-theor}
Throughout the note, we will only consider continuous functions between  discrete and/or  profinite sets.
We will use this convention with impunity, and write ``$\Hom$'' for the set of continuous homomorphisms, ``$\Fun$'' for the set of continuous functions, ``$f : A \rightarrow B$'' for a continuous map $A$ to $B$, etc. 

Let $K$ be a field and denote by $\Gal(K)$ its absolute Galois group.
Let $n$ be a positive integer which is relatively prime to $\Char K$ and assume that $\mu_n \subset K$.
Throughout the note we will deal with $\Z/n$ as our ring of coefficients, and we will denote it by $\Lambda$.

We recall that the mod-$n$ central descending series of a profinite group $\Gfrak$ is defined inductively as follows:
\begin{enumerate}
\item $\Gfrak^{(1)} = \Gfrak$.
\item $\Gfrak^{(i+1)} = [\Gfrak,\Gfrak^{(i)}] \cdot (\Gfrak^{(i)})^n$.
\end{enumerate}
In other words, for $i \geq 1$, $\Gfrak^{(i+1)}$ is precisely the left kernel of the canonical pairing:
\[ \Gfrak^{(i)} \times \Hom_{\Gfrak}(\Gfrak^{(i)},\Lambda) \rightarrow \Lambda \]
where $\Gfrak$ acts on $\Gfrak^{(i)}$ via conjugation, trivially on $\Lambda$, and $\Hom_\Gfrak(\Gfrak^{(i)},\Lambda)$ denotes the set of $\Gfrak$-equivariant homomorphisms $\Gfrak^{(i)} \rightarrow \Lambda$.
For $i \geq 1$, we denote $\Gfrak^{(i)}/\Gfrak^{(i+1)}$ by  $\gfrak_i(\Gfrak)$.
When $\Gfrak$ is understood from context, we will simplify the notation and denote $\gfrak_i(\Gfrak)$ simply by $\gfrak_i$.
Also in the case where $\Gfrak = \Gal(K)$, the absolute Galois group of a field $K$ as above, we denote $\gfrak_i(\Gal(K))$ by $\gfrak_i(K)$.
We will generally use additive notation for $\gfrak_i$ (the exception to this is in \S\ref{sec:cocycle-calculations}); in certain cases where we might want to consider $\gfrak_i$ as a multiplicative group, we will write $\Gfrak^{(i)}/\Gfrak^{(i+1)}$ instead of $\gfrak_i$.

For $\sigma,\tau \in \gfrak_1$, we define $[\sigma,\tau] := \tsigma^{-1}\ttau^{-1}\tsigma\ttau$ where $\tsigma,\ttau \in \Gfrak/\Gfrak^{(3)}$ are some lifts of $\sigma,\tau \in \gfrak_1 = \Gfrak/\Gfrak^{(2)}$.
Since $\Gfrak/\Gfrak^{(3)} \twoheadrightarrow \Gfrak/\Gfrak^{(2)}$ is a central extension, $[\sigma,\tau]$ doesn't depend on the choice of lifts $\tsigma,\ttau$ and thus $[\bullet,\bullet] : \gfrak_1 \times \gfrak_1 \rightarrow \gfrak_2$ is well-defined; it is well-known that $[\bullet,\bullet]$ is $\Lambda$-bilinear.
For $\sigma \in \gfrak_1$, we define $\sigma^\pi := \tsigma^n$ where $\tsigma \in \Gfrak/\Gfrak^{(3)}$ is some lift of $\sigma \in \gfrak_1 = \Gfrak/\Gfrak^{(2)}$.
As before, $\sigma^\pi$ doesn't depend on the choice of lift $\tsigma$ and thus $(\bullet)^\pi : \gfrak_1 \rightarrow \gfrak_2$ is a well-defined map.
The map $(\bullet)^\pi$ is not $\Lambda$-linear in general (e.g. if $n$ is even), but $2 \cdot (\bullet)^\pi$ is always $\Lambda$-linear.

Throughout the note we will work with a \emph{fixed} primitive $n$th root of unity $\omega \in \mu_n \subset K$.
This choice of $\omega$ yields an isomorphism of $G_K$-modules $\mu_n \cong \Lambda$, by sending $\omega^i \in \mu_n$ to $i \in \Lambda$, which we tacitly use throughout.
Kummer theory gives us a perfect pairing:
\[ \gfrak_1(K) \times K^\times/n  \rightarrow \mu_n \]
defined by $(\sigma,x) \mapsto \sigma\sqrt[n]{x}/\sqrt[n]{x} \in \mu_n$.
Therefore, we obtain an isomorphism $(\bullet)^\omega : \gfrak_1(K) \rightarrow \Hom(K^\times,\Lambda)$ using the Kummer pairing together with our fixed isomorphism $\mu_n \cong \Lambda$.
Namely, for $\sigma \in \gfrak_1(K)$ and $x \in K^\times$, one has $\sigma^\omega(x) = i$ if and only if $\sigma \sqrt[n]{x}/\sqrt[n]{x} = \omega^i$.

For a discrete set $S$ and a profinite set $P$, we denote by $\Fun(S,P)$ the set of functions $S \rightarrow P$.
Obviously, any function $S \rightarrow P$ is continuous.
Because $P$ is profinite, $\Fun(S,P)=P^S$ obtains a natural profinite topology and this agrees with the compact-open topology of $\Fun(S,P)$.

Let $f,g \in \Hom(K^\times,\Lambda)$ be given.
We now introduce two functions $\Phi^\omega(f,g) : K\smallsetminus\{0,1\} \rightarrow \Lambda^2$ and $\Psi^\omega(f) : K \smallsetminus\{0,1\} \rightarrow \Lambda^2$, defined as follows:
\begin{enumerate}
\item $\Phi^\omega(f,g)(x) = (f(x)g(1-x)-f(1-x)g(x), \ f(x)g(\omega)-f(\omega)g(x))$.
\item $\Psi^\omega(f)(x) = \left({n \choose 2}f(x)f(1-x), \ {n\choose 2}f(x)f(\omega)+f(x)\right)$.
\end{enumerate}
As above, we endow $\Fun(K \smallsetminus\{0,1\},\Lambda^2)$ with the compact-open topology and therefore $\Fun(K\smallsetminus\{0,1\},\Lambda^2) = (\Lambda^2)^{K\smallsetminus\{0,1\}}$ is canonically a profinite group.
Thus, it makes sense to speak about closed subgroups of $\Fun(K \smallsetminus \{0,1\},\Lambda^2)$.
We will denote by $\Ffrak_K^\omega$ the closed subgroup of $\Fun(K\smallsetminus\{0,1\},\Lambda^2)$ which is topologically generated by the functions $\Phi^\omega(f,g)$ and $\Psi^\omega(f)$ as $f,g \in \Hom(K^\times,\Lambda)$ vary.

Suppose that $L|K$ is an extension of fields.
Then the canonically induced map $\phi : \Gal(L) \rightarrow \Gal(K)$ restricts to homomorphisms $\phi^{(i)} : \Gal(L)^{(i)} \rightarrow \Gal(K)^{(i)}$ for all $i \geq 1$.
In particular, we obtain induced homomorphisms $\phi_i : \gfrak_i(L) \rightarrow \gfrak_i(K)$.
These induced maps, for $i = 1,2$, are clearly compatible with $[\bullet,\bullet]$ and $(\bullet)^\pi$ in the sense that, for all $\sigma,\tau \in \gfrak_1(K)$, one has $\phi_2[\sigma,\tau] = [\phi_1\sigma,\phi_1\tau]$ and $\phi_2(\sigma^\pi) = (\phi_1\sigma)^\pi$.

On the other hand, it is clear from the definition that the restriction map $\Fun(L\smallsetminus\{0,1\},\Lambda^2) \rightarrow \Fun(K\smallsetminus\{0,1\},\Lambda^2)$, induced by the inclusion $K \subset L$, restricts to a homomorphism $\res_K : \Ffrak_L^\omega \rightarrow \Ffrak_K^\omega$ which sends $\Phi^\omega(f,g)$ to $\Phi^\omega(f|_{K^\times},g|_{K^\times})$ and $\Psi^\omega(f)$ to $\Psi^\omega(f|_{K^\times})$, for all $f,g \in \Hom(L^\times,\Lambda)$.
Thus $\Ffrak_K^\omega$ is functorial in $K$, as long as we endow each such $K$ with the same fixed primitive $n$th root of unity $\omega$.

We are now ready to introduce the main theorem of the note, which relates the two profinite groups $\gfrak_2(K)$ and $\Ffrak_K^\omega$ in a functorial way.
\begin{theorem}[Main Theorem]
\label{thm:main-thm-ffrak}
Let $K$ be a field whose characteristic is prime to $n$ with $\mu_n\subset K$.
Choose $\omega \in \mu_n$ a fixed primitive $n$th root of unity.
Then there is a canonical isomorphism $\Omega_K : \gfrak_2(K) \rightarrow \Ffrak_K^\omega$ which satisfies:
\begin{enumerate}
\item $\Omega_K([\sigma,\tau]) = \Phi^\omega(\sigma^\omega,\tau^\omega)$.
\item $\Omega_K(\sigma^\pi) = \Psi^\omega(\sigma^\omega)$.
\end{enumerate}
Moreover, this isomorphism is functorial in $K$ in the sense that, if $L|K$ is a field extension, then the canonical map $\gfrak_2(L) \rightarrow \gfrak_2(K)$ corresponds to the restriction map $\Ffrak_L^\omega \rightarrow \Ffrak_K^\omega$ induced by $K \subset L$.
Namely, the following diagram commutes:
\[
\xymatrix{
\gfrak_2(L) \ar[d]_{\rm canon.}\ar[r]^{\Omega_L} &\Ffrak_L^\omega \ar[d]^{\res_K} \ar@{^{(}->}[r]^-{\rm incl.} & \Fun(L\smallsetminus\{0,1\},\Lambda^2)\ar[d]^{(K \subset L)^*}\\
\gfrak_2(K) \ar[r]_{\Omega_K} &\Ffrak_K^\omega \ar@{^{(}->}[r]_-{\rm incl.} & \Fun(K\smallsetminus\{0,1\},\Lambda^2) \\
}
\]
\end{theorem}
Theorem \ref{thm:main-thm-ffrak} shows that $\Ffrak_K^\omega$ yields a \emph{functorial formal description} of $\gfrak_2(K)$ in terms of \emph{functions} on $K\smallsetminus\{0,1\}$ with values in $\Lambda^2$.
Furthermore, this formal description is compatible with the Kummer isomorphism $(\bullet)^\omega : \gfrak_1(K) \rightarrow \Hom(K^\times,\Lambda)$, by identifying $[\sigma,\tau]$ resp. $\sigma^\pi$ with $\Phi^\omega(\sigma^\omega,\tau^\omega)$ resp. $\Psi^\omega(\sigma^\omega)$, for $\sigma,\tau \in \gfrak_1(K)$.
Theorem \ref{thm:main-thm-ffrak} is proved in sections \ref{sec:proof-theor-refthm-omegaK} and \ref{sec:proof-theor-functoriality}.

We also prove the following important corollary to Theorem \ref{thm:main-thm-ffrak} which gives an explicit and fairly restrictive condition on the types of relations that can occur between elements of the form $[\sigma,\tau]$ and $\sigma^\pi$ in the group $\gfrak_2(K)$.
\begin{corollary}[see Theorem \ref{thm:application-to-relations}]
\label{cor:cor-from-intro}
Let $K$ be a field whose characteristic is prime to $n$ with $\mu_{2n} \subset K$.
Choose $\omega \in \mu_n$ a fixed primitive $n$th root of unity.
Let $\sigma_i,\tau_i \in \gfrak_1(K)$ be a collection of elements which converges to $0$ in $\gfrak_1(K)$.
Let $a_i,b_i \in \Lambda$ be such that $2 \cdot a_i = \sigma_i^\omega(\omega)$ and $2\cdot b_i = \tau_i^\omega(\omega)$; since $\omega \in K^{\times 2}$ by assumption, such $a_i,b_i$ always exist.
Then the following are equivalent:
\begin{enumerate}
\item $\sum_i [\sigma_i,\tau_i] = \sum_i (b_i \cdot (2 \cdot \sigma_i^\pi) -a_i \cdot (2 \cdot \tau_i^\pi))$.
\item $\sum_i [\sigma_i,\tau_i] \in \langle 2 \cdot \sigma_i^\pi,2\cdot\tau_i^\pi\rangle_i$.
\item $\sum_i [\sigma_i,\tau_i] \in 2 \cdot \gfrak_1(K)^\pi$.
\end{enumerate}
\end{corollary}
In particular, the equivalence of (2) and (3) in Corollary \ref{cor:cor-from-intro} shows that a pair of elements $\sigma,\tau \in \gfrak_1(K)$ form a \emph{commuting-liftable} pair (i.e. $[\sigma,\tau] \in \langle 2 \cdot \sigma^\pi,2\cdot\tau^\pi\rangle$) if and only if $\sigma,\tau$ form a \emph{weakly-commuting-liftable} pair (i.e. $[\sigma,\tau] \in 2 \cdot \gfrak_2(K)^\pi$).
This observation thereby generalizes the various main results of \cite{Topaz2012c}; see \S\ref{sec:aappl-theor-comm} for a more detailed discussion.

\subsection*{Acknowledgments} The author warmly thanks Florian Pop and J\'an Min\'a\v{c} for several discussions which motivated this work.
The author also thanks Ido Efrat for his kind encouragement and for helpful suggestions regarding the exposition.

\section{Cohomological Minimal Presentations}
\label{sec:cohom-frat-covers}

Most of the cohomological results in this section are fairly well-known for pro-$p$ groups (and $n$ a power of $p$), due to the existence of \emph{minimal free pro-$p$ presentations} and standard Frattini-type arguments which stem from the Burnside basis theorem.
Such cohomological results, concerning pro-$p$ groups of finite rank, go back to {\sc Labute} \cite{Labute1967}; see also the exposition in \cite{Neukirch2008} around Propositions 3.9.13 and 3.9.14.
For pro-$p$ groups, these cohomological results were recently further studied by {\sc Efrat-Min\'a\v{c}} in \cite{Efrat2011b} Section 2, and in \cite{Efrat2011c} Section 10.
It is also important to note that in these approaches, in order to carry out the required cocycle calculations, one must choose a basis for $\H^1$, which yields a minimal generating set for the corresponding pro-$p$ group and fixes generators for the minimal free pro-$p$ presentation.

When generalizing to arbitrary profinite groups, however, one no longer has access to these minimal free presentations.
Also, the process of choosing a free presentation, and/or choosing a basis for $\H^1$, makes statements about functoriality difficult, or even impossible to prove in the usual category of profinite/pro-$p$ groups.

In this section we discuss the main new ingredient introduced this note.
We introduce a \emph{cohomological} analogue of a minimal free presentation for $\gfrak_2(\Gfrak)$, associated to an \emph{arbitrary} profinite group $\Gfrak$, which we denote by $\Sfr(\Gfrak)$.
One of the main benefits of our construction is that $\Sfr(\Gfrak)$ is purely functorial in $\Gfrak$.

Throughout this section we will work with a fixed profinite group $\Gfrak$.
Thus, we omit $\Gfrak$ from the notation and denote $\gfrak_i(\Gfrak)$ by $\gfrak_i$.
We will denote by $\H^*(\bullet) := \H^*(\bullet,\Lambda)$ the continuous-cochain cohomology of $\bullet$ with values in $\Lambda$.
Throughout, we will also assume that $\gfrak_1 := \Gfrak/\Gfrak^{(2)}$ is isomorphic to $(\Z/n)^I$ for some indexing set $I$. 
Note that if $K$ is a field whose characteristic is prime to $n$ with $\mu_n \subset K$, then $\gfrak_1(K) \cong (\Z/n)^I$, for some indexing set $I$, by Kummer theory; thus $\Gal(K)$ satisfies our added assumption.

Recall that Pontryagin duality yields a perfect pairing between $\H^1(\gfrak_1) = \H^1(\Gfrak)$ and $\gfrak_1$.
We define $\Kfr(\Gfrak) := \H^1(\gfrak_1)$.
As with $\gfrak_i$, we will denote $\Kfr(\Gfrak)$ by $\Kfr$ throughout this section.
Throughout the note, we will identify $\gfrak_1$ with $\Hom(\Kfr,\Lambda)$ via the canonical perfect pairing $\gfrak_1 \times \Kfr \rightarrow \Lambda$.

We recall that the Lyndon-Hochschild-Serre (LHS) spectral sequence associated to the extension $\Gfrak \rightarrow \Gfrak/\Gfrak^{(2)} = \gfrak_1$, 
\[ \H^i(\gfrak_1,\H^j(\Gfrak^{(2)})) \Rightarrow \H^{i+j}(\Gfrak), \]
combined with the fact that $\inf : \H^1(\gfrak_1) \rightarrow \H^1(\Gfrak)$ is an isomorphism, yields the following exact sequence:
\[ 0 \rightarrow \H^1(\Gfrak^{(2)})^{\gfrak_1} \xrightarrow{d_2} \H^2(\gfrak_1) \xrightarrow{\inf} \H^2(\Gfrak). \]
Above, $d_2$ denotes the differential on the ${\rm E}_2$-page in the LHS-spectral sequence:
\[ d_2 := d_2^{0,1} : {\rm E}_2^{0,1} \rightarrow {\rm E}_2^{2,0}. \]

Furthermore, we recall that the {\bf Bockstein} homomorphism $\beta : \H^1(\gfrak_1) \rightarrow \H^2(\gfrak_2)$ is the connecting homomorphism in the cohomological long exact sequence associated to the short exact sequence of coefficient modules (recall that $\Z/n = \Lambda$):
\[ 1 \rightarrow \Z/n \xrightarrow{n} \Z/n^2 \rightarrow \Z/n \rightarrow 1. \]

We denote by $\gfrak_1 \hotimes \gfrak_1$ the \emph{completed} tensor product of $\gfrak_1$ with itself.
In other words, 
\[ \gfrak_1 \hotimes \gfrak_1 = \varprojlim \gfrak_1/N_1 \otimes \gfrak_1/N_2 \]
where $N_1$ and $N_2$ vary over the finite index subgroups of $\gfrak_1$ and the tensor product is taken over $\Lambda$ by considering $\gfrak_1/N_i$ as $\Lambda$-modules in the obvious way.
One has a canonical map $\gfrak_1 \otimes \gfrak_1 \rightarrow \gfrak_1 \hotimes \gfrak_1$; if $\sigma,\tau \in \gfrak_1$, we will abuse the notation and write $\sigma \otimes \tau \in \gfrak_1 \hotimes \gfrak_1$ for the image of $\sigma \otimes \tau \in \gfrak_1 \otimes \gfrak_1$ under this map.
It is easy to see that $\gfrak_1 \hotimes \gfrak_1$ is Pontryagin dual to $\Kfr \otimes \Kfr$ via the pairing $(\gfrak_1 \hotimes \gfrak_1) \times (\Kfr \otimes \Kfr) \rightarrow \Lambda$ given by 
\[ (\sigma \otimes \tau,x \otimes y) \mapsto \sigma(x)\tau(y) \]
and extended linearly.
We will therefore sometimes denote elements $f$ of $\gfrak_1 \hotimes \gfrak_1$ as $\Lambda$-bilinear forms $\Kfr \times \Kfr \rightarrow \Lambda$, just as we denote elements $g$ of $\gfrak_1$ as homomorphisms $\Kfr \rightarrow \Lambda$.

Denote by $\Sfr(\Gfrak)$ the subset of $(\gfrak_1 \hotimes \gfrak_1) \times \gfrak_1$ defined by 
\[ \Sfr(\Gfrak) = \left\{\left(f,g\right) \ : \ \forall x \in \Kfr, \ f(x,x) = {n \choose 2} \cdot g(x) \right\}. \]
As with $\gfrak_i$ and $\Kfr$, we will omit $\Gfrak$ from the notation and denote $\Sfr(\Gfrak)$ by $\Sfr$ throughout this section.
Clearly $\Sfr$ is a closed subgroup of $(\gfrak_1 \hotimes \gfrak_1) \times \gfrak_1$.
We also observe that the image of the projection $\Sfr \rightarrow \gfrak_1 \hotimes \gfrak_1$ is the subgroup of alternating $\Lambda$-bilinear forms on $\Kfr \times \Kfr$ with values in $\Lambda$.

\begin{lemma}
\label{lemma-perf-pairing-H2}
In the notation above, one has a surjective map $(\Kfr \otimes \Kfr) \oplus \Kfr \rightarrow \H^2(\gfrak_1)$ defined by $(x \otimes y) \oplus z \mapsto x \cup y +\beta z$ (extended linearly).
Moreover, the kernel of this surjective map is generated by elements of the form $(x \otimes x) \oplus (-{n \choose 2} x)$ as $x \in \Kfr$ varies.
In particular, we obtain a canonical perfect pairing $(\bullet,\bullet)_{\Sfr} : \Sfr \times \H^2(\gfrak_1) \rightarrow \Lambda$, defined by 
\begin{enumerate}
\item $((f,g),x \cup y)_\Sfr = f(x,y)$
\item $((f,g),\beta z)_\Sfr =  g(z)$
\end{enumerate}
and extended linearly.
\end{lemma}
\begin{proof}
Observe that one has a perfect pairing
\[ ((\gfrak_1 \hotimes \gfrak_1) \times \gfrak_1) \times ((\Kfr \otimes \Kfr) \oplus \Kfr) \rightarrow \Lambda \]
defined by $((f,g),(x \otimes y) \oplus z) \mapsto f(x,y)+g(z)$.
Thus, it suffices to prove that the map $(\Kfr \otimes \Kfr) \oplus \Kfr \rightarrow \H^2(\gfrak_1)$, defined by $(x \otimes y)\oplus z \mapsto x\cup y+\beta z$, is surjective with kernel generated by elements of the form $(x \otimes x) \oplus (-{n \choose 2} \cdot x)$.
The statement concerning the induced perfect pairing $(\bullet,\bullet)_\Sfr$ would follow immediately from this and the definition of $\Sfr$.

Recall our overarching assumption that $\gfrak_1 \cong (\Z/n)^I$ for some indexing set $I$.
Now the required presentation of $\H^2((\Z/n)^I) \cong \H^2(\gfrak_1)$, as a quotient of $(\Kfr \otimes \Kfr) \oplus \Kfr$, is essentially well-known as it follows from the K\"{u}nneth Formula.
We give a brief argument below.

In the case where $\gfrak_1 \cong \Z/n$ is cyclic, it is well-known that: 
\begin{enumerate}
\item $\H^1(\Z/n) = \langle x \rangle \cong \Z/n$, where $x$ is the identity homomorphism $\Z/n \rightarrow \Z/n$,
\item $\H^2(\Z/n) = \langle \beta x \rangle \cong \Z/n$ and
\item $x \cup x = {n \choose 2} \beta x$ in $\H^2(\Z/n)$.
\end{enumerate}
Statements (1) and (2) above can be found in any standard book on group cohomology (see e.g. \cite{Neukirch2008} Chapter 1.7).
Statement (3) is also well-known, but we note that it follows from Theorem \ref{thm:cocycle-calculations} below if we take $\Gfrak = \hat \Z$.
Namely, $\H^2(\hat\Z) = 0$ and $\gfrak_1(\hat\Z)= \Z/n$; therefore $d_2 : \H^1(\hat\Z^{(2)})^{\Z/n} \rightarrow \H^2(\Z/n)$ is an isomorphism, etc.
Thus the lemma is proven in the case where $\# I = 1$.
The case where $I$ is finite now follows from the K\"{u}nneth formula for finite group cohomology, while the case where $I$ is arbitrary follows from the finite $I$ case using a limit argument.
\end{proof}

\begin{lemma}
\label{lemma-mod-n-frattini}
In the notation above, the surjective map $\Gfrak^{(2)} \twoheadrightarrow \gfrak_2$ yields a canonical perfect pairing:
\[ (\bullet,\bullet)_{\gfrak_2} : \gfrak_2 \times \H^1(\Gfrak^{(2)})^{\gfrak_1} \rightarrow \Lambda. \]
The dual of $-d_2 : \H^1(\Gfrak^{(2)})^{\gfrak_1} \rightarrow \H^2(\gfrak_1)$, via $(\bullet,\bullet)_{\gfrak_2}$ and the pairing $(\bullet,\bullet)_{\Sfr}$ of Lemma \ref{lemma-perf-pairing-H2}, yields a canonical surjective map $-d_2^\vee : \Sfr \rightarrow \gfrak_2$ which is defined by the equation $(-d_2^\vee(s),\phi)_{\gfrak_2} = (s,-d_2\phi)_{\Sfr}$.
\end{lemma}
\begin{proof}
Recall that $\H^1(\Gfrak^{(2)})^{\gfrak_1} = \Hom_{\Gfrak}(\Gfrak^{(2)},\Lambda)$.
Thus the first statement follows, using Pontryagin duality, from the definition of $\gfrak_2 = \Gfrak^{(2)}/\Gfrak^{(3)}$.
Indeed, recall that $\Gfrak^{(3)}$ is the left-kernel of the canonical pairing $\Gfrak^{(2)} \times \Hom_\Gfrak(\Gfrak^{(2)},\Lambda) \rightarrow \Lambda$.
The surjectivity of $-d_2^\vee : \Sfr \rightarrow \gfrak_2$ is immediate, by Pontryagin duality, since $-d_2 : \H^1(\Gfrak^{(2)})^{\gfrak_1} \rightarrow \H^2(\gfrak_1)$ is injective.
\end{proof}

The next theorem, and the corollary which follows it, are the essential steps in describing the surjective map $-d_2^\vee : \Sfr \rightarrow \gfrak_2$ of Lemma \ref{lemma-mod-n-frattini} in a more explicit way.
Before we state the theorem, we recall our convention that $\gfrak_1$ is identified with $\Hom(\Kfr,\Lambda)$; in particular, we write elements of $\gfrak_1$ as \emph{functions} on $\Kfr$ with values in $\Lambda$.

\begin{theorem}
\label{thm:cocycle-calculations}
In the notation above, let $u \in \H^1(\Gfrak^{(2)})^{\gfrak_1}$ be given and consider $\eta := d_2 u$.
Choose $x_i,y_i,z_j \in \Kfr$ with $\eta = \sum_i x_i \cup y_i + \sum_j \beta z_j$, (this is always possible by Lemma \ref{lemma-perf-pairing-H2}).
Then, for all $\sigma,\tau \in \gfrak_1$, the following hold:
\begin{enumerate}
\item $-([\sigma,\tau],u)_{\gfrak_2} =  \sum_i \sigma(x_i)\tau(y_i)-\sigma(y_i)\tau(x_i)$ and
\item $-(\sigma^\pi,u)_{\gfrak_2} = {n \choose 2}\cdot\sum_i \sigma(x_i)\sigma(y_i) + \sum_j \sigma(z_j)$.
\end{enumerate}
\end{theorem}
\begin{proof}
The proof of this theorem involves very explicit cocycle calculations which resemble the calculations in \cite{Neukirch2008} Propositions 3.9.13 and 3.9.14.
For sake of exposition, we defer the detailed proof of this theorem to the appendix: \S\ref{sec:cocycle-calculations}.
\end{proof}

Motivated by the viewpoint that $\Sfr$ should act as a sort-of free presentation for $\gfrak_2$, we are interested in finding elements of $\Sfr$ which map to $[\sigma,\tau]$ and $\sigma^\pi$ via the surjective map $-d_2^\vee : \Sfr \rightarrow \gfrak_2$ described in Lemma \ref{lemma-mod-n-frattini}.
Therefore, we now introduce two types of elements in $\Sfr$ which will work in general:
\begin{enumerate}
\item Let $\sigma,\tau \in \gfrak_1$ be given. Then $(\sigma \otimes \tau - \tau \otimes \sigma,0) \in \Sfr$. 
\item Let $\sigma \in \gfrak_1$ be given. Then $({n \choose 2}\sigma \otimes \sigma,\sigma) \in \Sfr$.
\end{enumerate}
The next corollary to Theorem \ref{thm:cocycle-calculations} shows that these special elements of $\Sfr$ will work for our purposes by mapping to $[\sigma,\tau]$ and $\sigma^\pi$.

\begin{corollary}
\label{cor:frattini-presentation}
In the notation above, consider the surjective map $-d_2^\vee :  \Sfr \rightarrow \gfrak_2$ of Lemma \ref{lemma-mod-n-frattini}, which is defined by $(-d_2^\vee(s),\phi)_{\gfrak_2} = (s,-d_2\phi)_\Sfr$.
Then the following hold:
\begin{enumerate}
\item For $\sigma,\tau \in \gfrak_1$, $-d_2^\vee(\sigma\otimes\tau-\tau\otimes\sigma,0) = [\sigma,\tau]$.
\item For $\sigma \in \gfrak_1$, $-d_2^\vee({n \choose 2} \sigma \otimes \sigma,\sigma) = \sigma^\pi$.
\end{enumerate}
\end{corollary}
\begin{proof}
Let $\sigma,\tau \in \gfrak_1$ be given.
It suffices to prove that for any $u \in \H^1(\Gfrak^{(2)})^{\gfrak_1}$, the following two equalities hold:
\begin{enumerate}
\item $([\sigma,\tau],u)_{\gfrak_2} = u([\sigma,\tau]) = ((\sigma\otimes\tau-\tau\otimes\sigma,0),-d_2(u))_\Sfr$ and 
\item $(\sigma^\pi,u)_{\gfrak_2} = u(\sigma^\pi) = (({n\choose 2}\sigma\otimes\sigma,\sigma),-d_2(u))_\Sfr$.
\end{enumerate} 
But this is precisely Theorem \ref{thm:cocycle-calculations} combined with the definition of $(\bullet,\bullet)_\Sfr$ from Lemma \ref{lemma-perf-pairing-H2}.
\end{proof}

Let $R$ be an arbitrary \emph{subset} of $\H^2(\gfrak_1)$ and consider the induced restriction homomorphism 
\[ \res_R : \Sfr \xrightarrow{s \mapsto (s,\bullet)_\Sfr}\Hom(\H^2(\gfrak_1),\Lambda) \xrightarrow{{\rm restriction}} \Fun(R,\Lambda). \]
We denote the image of $\res_R$ by $\Sfr_R$.
As before, we will keep $\Gfrak$ implicit in the notation for $\Sfr_R$.
Note that the kernel $R^\perp$ of $\Sfr \rightarrow \Sfr_R$ is precisely the annihilator of $R$ with respect to $(\bullet,\bullet)_\Sfr$:
\[ R^\perp = \{ s \in \Sfr \ : \ \forall r \in R, \ (s,r)_{\Sfr} = 0 \}. \]

\begin{theorem}
\label{thm:main-general-groups}
In the notation above, let $R$ be an arbitrary generating set for $\ker(\H^2(\gfrak_1) \rightarrow \H^2(\Gfrak))$.
Then $R$ induces a canonical isomorphism $\Omega_R : \gfrak_2 \rightarrow \Sfr_R$ defined as follows:
\begin{enumerate}
\item For $\sigma,\tau \in \gfrak_1$, $\Omega_R([\sigma,\tau]) = \res_R(\sigma\otimes\tau-\tau\otimes\sigma,0)$.
\item For $\sigma \in \gfrak_1$, $\Omega_R(\sigma^\pi) = \res_R({n \choose 2} \sigma \otimes \sigma,\sigma)$.
\end{enumerate}
\end{theorem}
\begin{proof}
Recall the existence of the following exact sequence:
\[ 0 \rightarrow \H^1(\Gfrak^{(2)})^{\gfrak_1} \xrightarrow{d_2} \H^2(\gfrak_1) \xrightarrow{\inf} \H^2(\Gfrak) \]
and that, by Lemma \ref{lemma-mod-n-frattini}, the dual of $-d_2$ yields a surjective map $-d_2^\vee : \Sfr \rightarrow \gfrak_2$.
Also recall that, by Corollary \ref{cor:frattini-presentation}, $-d_2^\vee$ satisfies the following equalities:
\begin{enumerate}
\item $-d_2^\vee(\sigma\otimes\tau-\tau\otimes\sigma,0) = [\sigma,\tau].$
\item $-d_2^\vee({n \choose 2}\sigma\otimes\sigma,\sigma) = \sigma^\pi.$
\end{enumerate}
By the definition of $\Sfr_R$, one has a canonical surjective map $\res_R : \Sfr \rightarrow \Sfr_R$.
Thus, it suffices to prove that the kernel of $-d_2^\vee : \Sfr \rightarrow \gfrak_2$ is the same as the kernel of $\res_R : \Sfr \rightarrow \Sfr_R$.
Indeed, we then can define $\Omega_R(\delta) := \res_R(\delta_1)$ where $\delta_1 \in \Sfr$ is chosen with $-d_2^\vee(\delta_1) = \delta$.

As observed above, the kernel of $\res_R : \Sfr \rightarrow \Sfr_R$ is precisely $R^\perp$, the annihilator of $R$ with respect to the pairing $(\bullet,\bullet)_{\Sfr}$.
Since $R$ generates the image of $d_2 : \H^1(\Gfrak^{(2)})^{\gfrak_1} \rightarrow \H^2(\gfrak_1)$, we see that $R^\perp$ is precisely $(\imm d_2)^\perp = (\imm(-d_2))^\perp$.
Since the map $-d_2^\vee : \Sfr \rightarrow \gfrak_2$ was defined to be the dual of $-d_2$, we obtain our claim: that the kernels of $-d_2^\vee$ and $\res_R$ are identical.
\end{proof}

\begin{remark}[Functoriality]
\label{remark:functoriality-general}
The isomorphism $\Omega_R$ is functorial in the following sense.
Suppose that $\phi : \Gfrak \rightarrow \Hfrak$ is a homomorphism of profinite groups.
Then $\phi$ induces canonical homomorphisms $\phi^{(i)} : \Gfrak^{(i)} \rightarrow \Hfrak^{(i)}$ for each $i \geq 1$.
Therefore, $\phi$ induces canonical homomorphisms $\phi_i : \gfrak_i(\Gfrak) \rightarrow \gfrak_i(\Hfrak)$ which are compatible with $\phi^{(i)}$.
In particular, for all $\sigma,\tau \in \gfrak_1(\Gfrak)$, one has $\phi_2[\sigma,\tau] = [\phi_1\sigma,\phi_1\tau]$ and $\phi_2(\sigma^\pi) = (\phi_1\sigma)^\pi$.
Moreover, the maps $\phi$ and $\phi_1$ yield the following commutative diagram:
\[ 
\xymatrix{
\H^2(\gfrak_1(\Gfrak)) \ar[r]^-{\inf} & \H^2(\Gfrak) \\
\H^2(\gfrak_1(\Hfrak)) \ar[u]^{\phi_1^*} \ar[r]_-{\inf} & \H^2(\Hfrak) \ar[u]_{\phi^*}
}
\]

Suppose that $R(\Gfrak)$ denotes a generating set for $\ker(\H^2(\gfrak_1(\Gfrak)) \rightarrow \H^2(\Gfrak))$ and $R(\Hfrak)$ denotes a generating set for $\ker(\H^2(\gfrak_1(\Hfrak)) \rightarrow \H^2(\Hfrak))$, in such a way so that the homomorphism $\phi_1^* : \H^2(\gfrak_1(\Hfrak)) \rightarrow \H^2(\gfrak_1(\Gfrak))$ restricts to a map $R(\Hfrak) \rightarrow R(\Gfrak)$.
We will show that $\phi_1^*$ induces a canonical map $\phi_1^{**}:\Sfr_{R(\Gfrak)} \rightarrow \Sfr_{R(\Hfrak)}$ which is compatible with $\phi_2$ via Theorem \ref{thm:main-general-groups}.
In other words, we will show that the following diagram is commutative:
\begin{eqnarray}
\label{eqn:main-functr-diagram}
\xymatrix{
\gfrak_2(\Gfrak) \ar[r]^{\Omega_{R(\Gfrak)}} \ar[d]_{\phi_2} & \Sfr_{R(\Gfrak)}\ar[d]^{\phi_1^{**}} \\
\gfrak_2(\Hfrak) \ar[r]_{\Omega_{R(\Hfrak)}} & \Sfr_{R(\Hfrak)}
}
\end{eqnarray}
and that it is compatible with $[\bullet,\bullet]$ and $(\bullet)^\pi$ similarly to Theorem \ref{thm:main-general-groups}.

To see this, first observe that the dual of $\phi_1^* : \H^2(\gfrak_1(\Hfrak)) \rightarrow \H^2(\gfrak_1(\Gfrak))$, via $(\bullet,\bullet)_{\Sfr}$, yields a homomorphism $(\phi_1^*)^\vee : \Sfr(\Gfrak) \rightarrow \Sfr(\Hfrak)$.
Furthermore, since $\phi_1^*$ restricts to a function $R(\Hfrak) \rightarrow R(\Gfrak)$, we obtain a commutative diagram:
\begin{eqnarray}
\label{eqn:res-with-phi}
\xymatrix{
\Sfr(\Gfrak) \ar[d]_{(\phi_1^*)^\vee} \ar[r]^-{\res_{R(\Gfrak)}} & \Fun(R(\Gfrak),\Lambda) \ar[d]^{\phi_1^{**}}  \\
\Sfr(\Hfrak) \ar[r]_-{\res_{R(\Hfrak)}} & \Fun(R(\Hfrak),\Lambda)
}
\end{eqnarray}
Thus $\phi_1^{**}$ in (\ref{eqn:res-with-phi}) restricts to a canonical homomorphism $\phi_1^{**} : \Sfr_{R(\Gfrak)} \rightarrow \Sfr_{R(\Hfrak)}$ since these groups are the images of the horizontal maps in diagram (\ref{eqn:res-with-phi}).
Namely, we obtain a commutative diagram:
\begin{eqnarray}
\label{eqn:sfr-with-sfr-sub-r}
\xymatrix{
\Sfr(\Gfrak) \ar[d]_{(\phi_1^*)^\vee} \ar@{->>}[r]^-{\res_{R(\Gfrak)}} & \Sfr_{R(\Gfrak)} \ar[d]^{\phi_1^{**}}  \\
\Sfr(\Hfrak) \ar@{->>}[r]_-{\res_{R(\Hfrak)}} & \Sfr_{R(\Hfrak)}
}
\end{eqnarray}
where the horizontal arrows are surjective.

On the other hand, the functoriality of the LHS-spectral sequence yields the following commutative diagram:
\begin{eqnarray}
\label{eqn:LHS-commutative-square}
\xymatrix{
\H^1(\Gfrak^{(2)})^{\gfrak_1(\Gfrak)}  \ar[r]^{-d_2} & \H^2(\gfrak_1(\Gfrak)) \\
\H^1(\Hfrak^{(2)})^{\gfrak_1(\Hfrak)} \ar[u]^{(\phi^{(2)})^*} \ar[r]_{-d_2} & \H^2(\gfrak_1(\Hfrak)) \ar[u]_{\phi_1^*}
}
\end{eqnarray}
We observe that $(\phi^{(2)})^*$ in diagram (\ref{eqn:LHS-commutative-square}) is dual to $\phi_2$ via the pairing $(\bullet,\bullet)_{\gfrak_2}$ (see Lemma \ref{lemma-mod-n-frattini}).
Thus, dualizing diagram (\ref{eqn:LHS-commutative-square}) via the pairings $(\bullet,\bullet)_\Sfr$ and $(\bullet,\bullet)_{\gfrak_2}$ as in Lemma \ref{lemma-mod-n-frattini}, we obtain the commutative diagram:
\begin{eqnarray}
\label{eqn:dual-of-LHS}
\xymatrix{
\Sfr(\Gfrak) \ar[d]_{(\phi_1^*)^\vee} \ar@{->>}[r]^-{-d_2^\vee} & \gfrak_2(\Gfrak) \ar[d]^{\phi_2}  \\
\Sfr(\Hfrak) \ar@{->>}[r]_-{-d_2^\vee} & \gfrak_2(\Hfrak)
}
\end{eqnarray}
whose horizontal arrows are surjective.

Lastly, the commutativity of diagram (\ref{eqn:main-functr-diagram}) follows from the fact that the kernel of $\res_{R(\bullet)}$ from diagram (\ref{eqn:sfr-with-sfr-sub-r}) is precisely the same as the kernel of $-d_2^\vee : \Sfr(\bullet) \rightarrow \gfrak_2(\bullet)$ from diagram (\ref{eqn:dual-of-LHS}), for $\bullet = \Gfrak,\Hfrak$, along with the definition of $\Omega_{R(\bullet)}$; see the proof of Theorem \ref{thm:main-general-groups}.
Also, it follows from Corollary \ref{cor:frattini-presentation} and the proof of Theorem \ref{thm:main-general-groups} that for all $\sigma,\tau \in \gfrak_1(\Gfrak)$:
\begin{enumerate}
\item $\phi_1^{**}(\Omega_{R(\Gfrak)}[\sigma,\tau]) = \Omega_{R(\Hfrak)}(\phi_2[\sigma,\tau]) = \Omega_{R(\Hfrak)}([\phi_1 \sigma,\phi_1\tau])$ and
\item $\phi_1^{**}(\Omega_{R(\Gfrak)}(\sigma^\pi)) = \Omega_{R(\Hfrak)}(\phi_2(\sigma^\pi)) = \Omega_{R(\Hfrak)}((\phi_1\sigma)^\pi)$.
\end{enumerate}
\end{remark}

\section{The Heisenberg Group}
\label{sec:an-exampl-heis}

In this section we will explore an explicit example: the Heisenberg Group over $\Lambda$.
In particular, we will explicitly work through the consequences of Theorem \ref{thm:main-general-groups} resp. Remark \ref{remark:functoriality-general} for the Heisenberg group resp. homomorphisms to the Heisenberg group.

\subsection{Recalling Facts about the Heisenberg Group}
We denote by $\Hcal_\Lambda$ the Heisenberg group over $\Lambda$.
Namely, $\Hcal_\Lambda$ is the set of all $3 \times 3$ upper-triangular matrices with coefficients in the ring $\Lambda$, whose diagonal entries are $1$.
To simplify the notation we denote elements of $\Hcal_\Lambda$ as follows:
\[ h(a,b;c) := \left( \begin{array}{ccc} 1 & a & c \\ 0 & 1 & b \\ 0 & 0 & 1 \end{array}\right). \]
Thus, multiplication in $\Hcal_\Lambda$ works as follows:
\[ h(a,b;c) \cdot h(a',b';c') = h(a+a',b+b';c+c'+ab'). \]
A simple calculation shows that $\Hcal_\Lambda^{(3)}$ is trivial, the map $h(0,0;c) \mapsto c$ yields an isomorphism $\gfrak_2(\Hcal_\Lambda) \cong \Lambda$, and the map $h(a,b;c) \rightarrow (a,b)$ yields an isomorphism $\gfrak_1(\Hcal_\Lambda) \cong \Lambda^2$.
We will denote the image of $h(a,b;c)$ in $\gfrak_1(\Hcal_\Lambda)$ by $h_1(a,b)$, and we will denote $h(0,0;c)$ by $h_2(c)$.
A simple calculation shows that:
\begin{enumerate}
\item $[h_1(a,b),h_1(a',b')] = h_2(ab'-a'b)$.
\item $h_1(a,b)^\pi = h_2({n \choose 2} a b)$.
\end{enumerate}

Denote by $e_1,e_2$ the basis for $\Kfr(\Hcal_\Lambda) = \H^1(\gfrak_1(\Hcal_\Lambda))$ which is dual to $h_1(1,0),h_1(0,1)$.
In other words, the isomorphism $h_1^{-1} : \gfrak_1(\Hcal_\Lambda) \rightarrow \Lambda^2$ is precisely the map $(e_1,e_2)$.
This allows us to describe the equivalence class of $\Hcal_\Lambda$ as an extension of $\gfrak_1(\Hcal_\Lambda) \cong \Lambda^2$ by $\gfrak_2(\Hcal_\Lambda) \cong \Lambda$ as follows.
\begin{lemma}
\label{lemma:heisenberg-cup-product}
Consider $\Hcal_\Lambda$ as an extension of $\gfrak_1(\Hcal_\Lambda)$ by $\Lambda$ via the isomorphism $h_2 :  \Lambda \rightarrow \gfrak_2(\Hcal_\Lambda)$.
Then the class of $\Hcal_\Lambda$ in $\H^2(\gfrak_1(\Hcal_\Lambda))$ is given by $e_1 \cup e_2$.
\end{lemma}
\begin{proof}
The $2$-cocycle representing the equivalence class of $\Hcal_\Lambda$ in $\H^2(\gfrak_1(\Hcal_\Lambda))$ is defined as follows:
\[ (h_1(a,b),h_1(a',b')) \mapsto h(a,b;0)\cdot h(a',b';0) \cdot h(-(a+a'),-(b+b');0)^{-1} \in \gfrak_2(\Hcal_\Lambda). \]
A simple calculation shows that 
\[ h(a,b;0)\cdot h(a',b';0) \cdot h(-(a+b'),-(b+b');0)^{-1} = h(0,0;ab'). \]
Thus, a 2-cocycle representing the class of $\Hcal_\Lambda$ in $\H^2(\gfrak_1(\Hcal_\Lambda))$ is given by 
\[\xi : (h_1(a,b),h_1(a',b')) \mapsto ab'.\]
Furthermore, since $e_i(h_1(a_1,a_2)) = a_i$ for $i = 1,2$, we see that the class of $\Hcal_\Lambda$ is indeed equal to $e_1 \cup e_2 \in \H^2(\gfrak_1(\Hcal_\Lambda))$ since $e_1 \cup e_2$ is represented by the same cocycle $\xi$. 
\end{proof}

By Lemma \ref{lemma:heisenberg-cup-product}, we see that $e_1 \cup e_2 \in \ker(\H^2(\gfrak_1(\Hcal_\Lambda)) \rightarrow \H^2(\Hcal_\Lambda))$ (see e.g. Proposition \ref{prop:embedding-problems-h2} below).
Moreover, the isomorphism $h_2^{-1} : \gfrak_2(\Hcal_\Lambda) \rightarrow \Lambda$, which we can consider as an element of $\H^1(\Hcal_\Lambda^{(2)})^{\gfrak_1(\Hcal_\Lambda)}$, satisfies $-d_2(h_2^{-1}) = e_1 \cup e_2$ by Theorem \ref{thm:cocycle-calculations} and the calculations above.
Since $h_2^{-1}$ generates $\H^1(\Hcal_\Lambda^{(2)})^{\gfrak_1(\Hcal_\Lambda)}$, we see that $e_1 \cup e_2$ is a generator for $\ker(\H^2(\gfrak_1(\Hcal_\lambda)) \rightarrow \H^2(\Hcal_\Lambda))$.
In particular, Theorem \ref{thm:main-general-groups} yields an isomorphism $\Omega_{e_1 \cup e_2} : \gfrak_2(\Hcal_\Lambda) \rightarrow \Lambda$ which satisfies the following properties for $\sigma,\tau \in \gfrak_1(\Hcal_\Lambda)$:
\begin{enumerate}
\item $\Omega_{e_1\cup e_2}([\sigma,\tau]) = \sigma(e_1)\tau(e_2)-\sigma(e_2)\tau(e_1)$. 
\item $\Omega_{e_1\cup e_2}(\sigma^\pi) = {n\choose 2}\sigma(e_1)\sigma(e_2)$.
\end{enumerate}

\subsection{Recalling Facts About Embedding Problems}
An embedding problem $\mathcal{E}$ is a pair of homomorphisms of profinite groups with the same codomain:
\[ \mathcal{E} := (\widetilde \Hfrak \rightarrow \Hfrak \ ; \ \Gfrak \rightarrow \Hfrak). \]
A solution to the embedding problem $\mathcal{E}$ is a homomorphism $\Gfrak \rightarrow \widetilde\Hfrak$ which makes the following diagram commute:
\[
\xymatrix{
{} & \Gfrak  \ar[d] \ar@{.>}[ld]\\
\widetilde\Hfrak \ar[r] & \Hfrak
}
\]
If $\widetilde \Hfrak \rightarrow \Hfrak$ is a surjective homomorphism, which we consider as a group extension of $\Hfrak$ by $\ker(\widetilde\Hfrak \rightarrow \Hfrak)$, we denote the embedding problem $(\widetilde\Hfrak \rightarrow \Hfrak; \Gfrak \rightarrow \Hfrak)$ simply by $(\widetilde \Hfrak;\Gfrak \rightarrow \Hfrak)$.

\begin{proposition}
\label{prop:embedding-problems-h2}
Let $\xi \in \H^2(\Hfrak,A)$ represent two equivalent group extensions $\widetilde \Hfrak, \widetilde \Hfrak'$ of $\Hfrak$ by $A$, and let $\phi : \Gfrak \rightarrow \Hfrak$ be a homomorphism.
Then the following conditions are equivalent:
\begin{enumerate}
\item $(\widetilde \Hfrak;\phi)$ has a solution.
\item $(\widetilde \Hfrak';\phi)$ has a solution.
\item $\xi$ is in the kernel of $\phi^* : \H^2(\Hfrak,A) \rightarrow \H^2(\Gfrak,A)$.
\end{enumerate} 
\end{proposition}
\begin{proof}
Since $\widetilde \Hfrak$ and $\widetilde\Hfrak'$ are equivalent, one has a commutative diagram with exact rows:
\[ 
\xymatrix{
1 \ar[r] & A \ar[r] \ar@{=}[d]& \widetilde \Hfrak \ar[d]^-{\cong}\ar[r] & \Hfrak \ar@{=}[d]\ar[r] & 1 \\
1 \ar[r] & A \ar[r] & \widetilde \Hfrak' \ar[r] & \Hfrak \ar[r] & 1 \\
}
\]
From this it is clear that (1) and (2) are equivalent.

Assume that $(\widetilde\Hfrak;\phi)$ has a solution, say $\tilde\phi$. 
Recall that $\phi^*\xi$ has a representative group extension given by the fiber product $\widetilde \Hfrak \times_\Hfrak \Gfrak$.
The universal property of fiber products applied to $\tilde\phi$ yields a splitting of the short exact sequence
\[  1 \rightarrow A \rightarrow \widetilde \Hfrak \times_\Hfrak \Gfrak \rightarrow \Gfrak \rightarrow 1. \]
Thus $\phi^*\xi$ is trivial as it is represented by a split extension.

Similarly, if $\phi^*\xi$ is trivial, then there is a splitting $\Gfrak \rightarrow \widetilde\Hfrak \times_\Hfrak \Gfrak$ of the short exact sequence above.
Composing this splitting with the canonical map $\widetilde \Hfrak \times_\Hfrak \Gfrak \rightarrow \widetilde \Hfrak$ yields a solution to the embedding problem $(\widetilde \Hfrak;\phi)$.
\end{proof}

By Proposition \ref{prop:embedding-problems-h2}, it makes sense to define embedding problems of the form $(\xi;\phi)$, for $\xi \in \H^2(\Hfrak,A)$ and $\phi : \Gfrak \rightarrow \Hfrak$.
Saying that $(\xi;\phi)$ has a solution is equivalent to saying that $\phi^*\xi = 0$ in $\H^2(\Gfrak)$.

\begin{proposition}
\label{prop:heisenberg-embedding-problems}
Let $\Gfrak$ be a profinite group.
Let $x,y \in \Kfr(\Gfrak) = \H^1(\gfrak_1(\Gfrak))$ be given and consider the homomorphism $(x,y) : \gfrak_1(\Gfrak) \rightarrow \Lambda^2$ induced by $x,y$.
Then the embedding problem 
\[ \mathcal{E}_{x,y} := (\Hcal_\Lambda \twoheadrightarrow \gfrak_1(\Hcal_\Lambda) \xrightarrow{(e_1,e_2)} \Lambda^2 \ ; \ \Gfrak/\Gfrak^{(3)} \twoheadrightarrow \gfrak_1(\Gfrak) \xrightarrow{(x,y)} \Lambda^2) \]
has a solution if and only if the image of $x \cup y$ vanishes in $\H^2(\Gfrak)$.
\end{proposition}
\begin{proof}
First assume that $\mathcal{E}_{x,y}$ has a solution.
The fact that $x \cup y$ vanishes in $\H^2(\Gfrak)$ follows from Proposition \ref{prop:embedding-problems-h2} and the fact that $e_1 \cup e_2$ generates the kernel of $\H^2(\gfrak_1(\Hcal_\Lambda)) \rightarrow \H^2(\Hcal_\Lambda)$ (see the discussion which follows Lemma \ref{lemma:heisenberg-cup-product}).
Indeed, Proposition \ref{prop:embedding-problems-h2} implies that the image of $x \cup y$ vanishes in $\H^2(\Gfrak/\Gfrak^{(3)})$, and thus the image of $x \cup y$ also vanishes in $\H^2(\Gfrak)$.

Conversely, assume that the image of $x \cup y$ vanishes in $\H^2(\Gfrak)$.
Then by Proposition \ref{prop:embedding-problems-h2} and Lemma \ref{lemma:heisenberg-cup-product}, we obtain a solution $\phi_0 : \Gfrak \rightarrow \Hcal_\Lambda$ for the embedding problem:
\[ \widetilde{\mathcal{E}}_{x,y} := (\Hcal_\Lambda \twoheadrightarrow \gfrak_1(\Hcal_\Lambda) \xrightarrow{(e_1,e_2)} \Lambda^2 \ ; \ \Gfrak \twoheadrightarrow \gfrak_1(\Gfrak) \xrightarrow{(x,y)} \Lambda^2). \]
However, since $\Hcal_\Lambda^{(3)} = 0$, the solution $\phi_0$ descends to a solution $\phi : \Gfrak/\Gfrak^{(3)} \rightarrow \Hcal_\Lambda$ for our original embedding problem $\mathcal{E}_{x,y}$.
\end{proof}

\subsection{Relations Induced by the Heisenberg Group}

In this section we show how one can determine certain kinds of relations between elements of the form $[\sigma,\tau]$ and $\sigma^\pi$, by looking at homomorphisms $\gfrak_1 \rightarrow \Lambda^2$ which lift to homomorphisms $\Gfrak/\Gfrak^{(3)} \rightarrow \Hcal_\Lambda$.
We briefly recall the notation introduced in Remark \ref{remark:functoriality-general}: for a homomorphism $\phi : \Gfrak \rightarrow \Hfrak$, one obtains induced homomorphisms $\phi_i : \gfrak_i(\Gfrak) \rightarrow \gfrak_i(\Hfrak)$ which satisfy $\phi_2[\sigma,\tau] = [\phi_1\sigma,\phi_1\tau]$ and $\phi_2(\sigma^\pi) = (\phi_1(\sigma))^\pi$ for all $\sigma,\tau \in \gfrak_1(\Gfrak)$.

Let us first make a couple of observations concerning convergence of elements in profinite groups.
Let $\Gfrak$ be an arbitrary profinite group and let $\sigma_i \in \Gfrak$ be given, with a possibly infinite indexing set for $i$.
Recall that the collection $(\sigma_i)_i$ \emph{converges to $1$} provided that, for all finite-index normal subgroups $N$ of $\Gfrak$, all but finitely many of the $\sigma_i$ are contained in $N$.
If $\Gfrak$ is an abelian profinite group, and $(\sigma_i)_i$ converges to $1$, then one obtains a well-defined element $\prod_i \sigma_i$ of $\Gfrak$, since the product is well-defined in every finite quotient of $\Gfrak$.

Assume that one is given a collection $(\sigma_i)_i$ of elements $\sigma_i \in \gfrak_1(\Gfrak)$.
Then $(\sigma_i)_i$ converges to $0$ in $\gfrak_1(\Gfrak)$ if and only if for all $x \in \Kfr(\Gfrak) = \H^1(\gfrak_1(\Gfrak))$, all but finitely many of the $\sigma_i(x)$ are trivial.
Furthermore, since $[\bullet,\bullet]$ and $(\bullet)^\pi$ are continuous, if $(\sigma_i,\tau_i,\gamma_j)_{i,j}$ converges to $0$ in $\gfrak_1(\Gfrak)$, then $([\sigma_i,\tau_i],\gamma_j^\pi)_{i,j}$ converges to $0$ in $\gfrak_2(\Gfrak)$.

\begin{proposition}
\label{prop:heisenberg}
Let $\Gfrak$ be a profinite group with $\gfrak_1(\Gfrak)$ isomorphic to $(\Z/n)^I$ for some indexing set $I$.
Let $x,y \in \Kfr(\Gfrak) = \H^1(\gfrak_1(\Gfrak))$ be given and consider the homomorphism $(x,y) : \gfrak_1(\Gfrak) \rightarrow \Lambda^2$ induced by $x,y$.
Suppose furthermore that the equivalent conditions of Proposition \ref{prop:heisenberg-embedding-problems} hold, and that $\phi : \Gfrak/\Gfrak^{(3)} \rightarrow \Hcal_\Lambda$ is a solution to the embedding problem:
\[ \mathcal{E}_{x,y} := (\Hcal_\Lambda \twoheadrightarrow \gfrak_1(\Hcal_\Lambda) \xrightarrow{(e_1,e_2)} \Lambda^2 \ ; \ \Gfrak/\Gfrak^{(3)} \twoheadrightarrow \gfrak_1(\Gfrak) \xrightarrow{(x,y)} \Lambda^2). \]
Let $\sigma_i,\tau_i \in \gfrak_1(\Gfrak)$ be given, and assume that the collection $(\sigma_i,\tau_i)_i$ converges to $0$ in $\gfrak_1(\Gfrak)$.
Then the following are equivalent:
\begin{enumerate}
\item $\sum_i[\phi_1\sigma_i,\phi_1\tau_i] = 0$, as elements of $\gfrak_2(\Hcal_\Lambda)$.
\item $\sum_i (\sigma_i(x) \cdot \tau_i(y) - \sigma_i(y) \cdot \tau_i(x)) = 0$.
\end{enumerate}
\end{proposition}
\begin{proof}
Recall that $e_1 \cup e_2$ generates $\ker(\H^2(\gfrak_1(\Hcal_\Lambda)) \rightarrow \H^2(\Hcal_\Lambda))$.
Moreover, if we denote by $\phi_1$ the homomorphism $\gfrak_1(\Gfrak) \rightarrow \gfrak_1(\Hcal_\Lambda)$ defined by $\sigma \mapsto h_1(\sigma(x),\sigma(y))$, then the image of $e_1 \cup e_2$ under the induced map $\phi_1^* : \H^2(\gfrak_1(\Hcal_\Lambda)) \rightarrow \H^2(\gfrak_1(\Gfrak))$ is precisely $x \cup y$.
Our assumptions ensure that $x \cup y$ is contained in $\ker(\H^2(\gfrak_1(\Gfrak)) \rightarrow \H^2(\Gfrak))$.
Therefore, we can choose $R$, a set of generators for $\ker(\H^2(\gfrak_1(\Gfrak)) \rightarrow \H^2(\Gfrak))$, which contains $x \cup y$.

Consider the isomorphisms $\Omega_{e_1 \cup e_2} : \gfrak_2(\Hfrak_\Lambda) \rightarrow \Sfr_{e_1 \cup e_2}$ and $\Omega_R : \gfrak_2(\Gfrak) \rightarrow \Sfr_R$ of Theorem \ref{thm:main-general-groups}.
Since $\phi_1^*(e_1 \cup e_2) = x \cup y \in R$, we obtain a canonical restriction map $\phi_1^{**} : \Sfr_R \rightarrow \Sfr_{e_1 \cup e_2}$ which is compatible with the canonical map $\phi_2 : \gfrak_2(\Gfrak) \rightarrow \gfrak_2(\Hcal_\Lambda)$ as discussed in Remark \ref{remark:functoriality-general}.
Therefore, the restriction to $e_1 \cup e_2$, under the inclusion $e_1 \cup e_2 \mapsto x \cup y \subset R$, of the element
\[\Omega_R(\sum_i [\sigma_i,\tau_i]) = \sum_i \Omega_R([\sigma_i,\tau_i]) = \sum_i \res_R(\sigma_i \otimes\tau_i-\tau_i\otimes\sigma_i,0), \]
is trivial if and only if $\phi_2(\sum_i[\sigma_i,\tau_i]) = 0$ in $\gfrak_2(\Hcal_\Lambda)$.
In other words, $\phi_2(\sum_i[\sigma_i,\tau_i]) = 0$ if and only if 
\[ \sum_i \sigma_i(x)\tau_i(y)-\sigma_i(y)\tau_i(x) = 0. \]
This proves the proposition since $\phi_2(\sum_i[\sigma_i,\tau_i]) = \sum_i[\phi_1\sigma_i,\phi_1\tau_i]$.
\end{proof}

In \S\ref{sec:applications}, we will use Proposition \ref{prop:heisenberg} to explicitly describe certain kinds of relations among elements of the form $[\sigma,\tau]$ and $\sigma^\pi$ in $\gfrak_2(K)$ for a field $K$ whose characteristic is prime to $n$ with $\mu_n \subset K$; see Theorem \ref{thm:application-to-relations} for the details.

\section{Galois Groups}
\label{sec:proof-main-theorem}

The goal for this section will be to prove Theorem \ref{thm:main-thm-ffrak}.
Throughout this section $K$ will be a field whose characteristic is prime to $n$ with $\mu_n \subset K$ and $\omega \in \mu_n$ will be a fixed primitive $n$th root of unity.
In many cases we will denote $\H^*(\Gal(K),\bullet)$ by $\H^*(K,\bullet)$ as is usual; when we wish to emphasize the connection with previous notation, we will still use the notation $\H^*(\Gal(K),\bullet)$.
Continuing with the previous notation, our profinite group $\Gfrak$ will be $\Gal(K)$ in the remainder of the note.
Kummer theory yields a canonical perfect pairing
\[ \gfrak_1(K) \times K^\times/n \rightarrow \mu_n. \]
Thus, we will identify $K^\times/n$ with $\Kfr = \H^1(\gfrak_1(K))$ via our choice of $\omega \in \mu_n$, as follows.
For $x \in K^\times$, we denote by $x_\omega$ the homomorphism $\gfrak_1(K) \rightarrow \Lambda$ defined by $x_\omega(\sigma) = i$ if and only if $\sigma\sqrt[n]{x}/\sqrt[n]{x} = \omega^i$.
Thus the map $x \mapsto x_\omega$ yields an isomorphism $K^\times/n \rightarrow \Kfr = \H^1(\gfrak_1(K))$.
Lastly, as in \S\ref{sec:cohom-frat-covers}, we identify $\gfrak_1(K)$ canonically with $\Hom(\Kfr,\Lambda)$.
The connection with the notation of sections \ref{sec:notat-main-theor} and \ref{sec:cohom-frat-covers} is as follows: for $\sigma \in \gfrak_1(K)$ and $x \in K^\times$, one has $x_\omega \in \Kfr$ and $\sigma^\omega(x) = \sigma(x_\omega)$.

The following lemma calculates the Bockstein map in Galois cohomology.
This calculations seems to be fairly well-known: see \cite{Gille2006} Lemma 7.5.10 and/or \cite{Efrat2011b} Proposition 2.6.
We provide a precise proof of the lemma for the sake of completeness.
\begin{lemma}
\label{lem:bock-in-galois-cohom}
Let $K$ be a field whose characteristic is prime to $n$ with $\mu_n \subset K$ and let $\omega \in \mu_n$ be a chosen primitive $n$th root of unity.
Denote by $\inf_K$ the inflation map $\H^2(\gfrak_1(K)) \rightarrow \H^2(\Gal(K))$.
One has the following equality in $\H^2(\Gal(K))$ for all $x \in K^\times$:
\[ \inf_K(\beta x_\omega) = \inf_K(\omega_\omega \cup x_\omega). \]
\end{lemma}
\begin{proof}
We denote by $\delta_i : \H^1(K,\Z/n(i)) \rightarrow \H^2(K,\Z/n(i))$ the connecting homomorphism associated to the short exact sequence of $\Gal(K)$-modules:
\[ 1 \rightarrow \Z/n(i) \rightarrow \Z/n^2(i) \rightarrow \Z/n(i) \rightarrow 1. \]

Observe that the isomorphism $\H^i(K,\Lambda) \rightarrow \H^i(K,\Lambda(j))$, induced by $\omega$, is given by 
\[ \alpha \mapsto \alpha \cup (\underbrace{\omega \cup \cdots \cup \omega}_{j \ \text{times}}) \]
where we have $\omega \in \mu_n = \H^0(K,\Lambda(1))$.
Furthermore, for $x \in \mu_n = \H^0(K,\Lambda(1))$, we observe that 
\[ x_\omega \cup \omega = \delta_1(x) \]
since the Kummer homomorphism $K^\times = \H^0(K,(K^{\rm sep})^{\times}) \rightarrow \H^1(K,\Lambda(1))$ is precisely the connecting homomorphism associated to:
\[ 1 \rightarrow \mu_n \rightarrow (K^{\rm sep})^\times \xrightarrow{n} (K^{\rm sep})^\times \rightarrow 1. \]
To prove the lemma it suffices prove the following claim: for all $\alpha \in \H^1(K,\Lambda)$, one has
\[ \delta_0(\alpha) \cup \omega \cup \omega = \delta_1(\omega) \cup (\alpha \cup \omega). \]
Indeed, then one would have:
\begin{align*}
\delta_0(\alpha) \cup \omega \cup \omega &= \delta_1(\omega) \cup (\alpha \cup \omega) \\
&= (\omega_\omega \cup \omega) \cup (\alpha \cup \omega) \\
&= (\omega_\omega \cup \alpha) \cup \omega \cup \omega
\end{align*}
which implies that $\delta_0(\alpha) = \omega_\omega \cup \alpha$.

To prove this claim, we calculate using the well-known elementary identities involving cup products and connecting homomorphisms (see e.g. \cite{Neukirch2008} Proposition 1.4.3):
\begin{align*}
\delta_0(\alpha) \cup \omega \cup \omega &= \delta_1(\alpha \cup \omega) \cup \omega \\
&= \delta_1(\omega \cup \alpha) \cup \omega \\
&= \delta_1(\omega) \cup (\alpha \cup \omega).
\end{align*}
Thus we obtain the claim and the lemma follows.
\end{proof}

\begin{lemma}
\label{lem:ker-of-inf-galois}
Let $K$ be a field whose characteristic is prime to $n$ with $\mu_n \subset K$ and let $\omega \in \mu_n$ be a primitive $n$th root of unity.
Then the kernel of $\inf_K : \H^2(\gfrak_1(K)) \rightarrow \H^2(\Gal(K))$ is generated by elements of the form:
\begin{enumerate}
\item $x_\omega \cup (1-x)_\omega$
\item $x_\omega \cup \omega_\omega + \beta x_\omega$
\end{enumerate}
as $x \in K\smallsetminus \{0,1\}$ vary.
\end{lemma}
\begin{proof}
By the Merkurjev-Suslin theorem \cite{Merkurjev1982}, the composition 
\[ \H^1(\gfrak_1(K)) \otimes \H^1(\gfrak_1(K)) \xrightarrow{\cup} \H^2(\gfrak_1(K)) \xrightarrow{\inf} \H^2(\Gal(K)) \]
is surjective with kernel generated by elements of the form $x_\omega \otimes (1-x)_\omega$.
Thus, the claim easily follows from Lemma \ref{lemma-perf-pairing-H2} and Lemma \ref{lem:bock-in-galois-cohom}.
\end{proof}

\subsection{The Proof of Theorem \ref{thm:main-thm-ffrak} -- Construction of $\Omega_K$}
\label{sec:proof-theor-refthm-omegaK}
For simplicity, we denote $K\smallsetminus\{0,1\}$ by $\Kbb$.
Furthermore, consider the set
\[ R(K) := \{x_\omega \cup (1-x)_\omega\}_{x \in \Kbb} \cup \{x_\omega \cup \omega_\omega + \beta x_\omega \}_{x \in \Kbb} \subset H^2(\gfrak_1(K)). \]
Then one has a surjective function
\[ \Mcal_K : \Kbb \sqcup \Kbb \rightarrow R(K), \]
which is defined by sending $x$ in the first $\Kbb$ component to $x_\omega \cup (1-x)_\omega$ and $x$ in the second $\Kbb$ component to $x_\omega \cup \omega_\omega + \beta x_\omega$.
By Lemma \ref{lem:ker-of-inf-galois}, $R(K)$ is a generating set for $\ker(\H^2(\gfrak_1(K)) \rightarrow \H^2(\Gal(K))$.

The surjective function $\Mcal_K$ induces an \emph{injective} homomorphism:
\[ \Mcal_K^* : \Fun(R(K),\Lambda) \hookrightarrow \Fun(\Kbb \sqcup \Kbb,\Lambda) = \Fun(\Kbb,\Lambda)^2 = \Fun(\Kbb,\Lambda^2). \]
Explicitly, this homomorphism sends a function $f : R(K) \rightarrow \Lambda$ to the function $\Kbb \rightarrow \Lambda^2$ defined by:
\[x \mapsto (f(x_\omega \cup (1-x)_\omega), \ f(x_\omega \cup \omega + \beta x_\omega)). \]

Now we consider the inclusion induced by $\Omega_{R(K)}$ of Theorem \ref{thm:main-general-groups} composed with $\Mcal_K^*$ from above:
\[ \Omega_K :  \gfrak_2(K) \xrightarrow{\Omega_{R(K)}} \Sfr_{R(K)} \subset \Fun(R(K),\Lambda) \xrightarrow{\Mcal_K^*} \Fun(\Kbb,\Lambda^2). \]
This map will induce our required isomorphism $\Omega_K$ between $\gfrak_2(K)$ and $\Ffrak_K^\omega \subset \Fun(\Kbb,\Lambda^2)$.
Namely, it remains to prove that the image of $\Omega_K$, as defined above, is precisely $\Ffrak_K^\omega$.
In other words, now we need to trace through the maps to describe $\Omega_K([\sigma,\tau])$ and $\Omega_K(\sigma^\pi)$ as functions $\Kbb \rightarrow \Lambda^2$.

Let $\sigma,\tau \in \gfrak_1(K)$ be given.
We first recall that by Theorem \ref{thm:main-general-groups}:
\begin{enumerate}
\item $\Omega_{R(K)}([\sigma,\tau]) = \res_{R(K)}(\sigma\otimes\tau-\tau\otimes\sigma,0)$.
\item $\Omega_{R(K)}(\sigma^\pi) = \res_{R(K)}({n \choose 2} \sigma \otimes \sigma,\sigma)$.
\end{enumerate}
Next we calculate $(\bullet,\bullet)_\Sfr$ of Lemma \ref{lemma-perf-pairing-H2}, applied to our special elements of $\Sfr$ and arbitrary elements of $R(K)$, recalling that $\sigma(x_\omega) = \sigma^\omega(x)$ for $\sigma \in \gfrak_1(K)$ and $x \in K^\times$:
\begin{enumerate}
\item $((\sigma\otimes\tau-\tau\otimes\sigma,0),x_\omega \cup (1-x)_\omega)_{\Sfr} = \sigma^\omega(x)\tau^\omega(1-x)-\tau^\omega(x)\sigma^\omega(1-x)$ and 
\item $((\sigma\otimes\tau-\tau\otimes\sigma,0),x_\omega \cup \omega_\omega + \beta x_\omega)_{\Sfr} = \sigma^\omega(x)\tau^\omega(\omega)-\tau^\omega(x)\sigma^\omega(\omega)$.
\item $(({n \choose 2}\sigma \otimes \sigma,\sigma),x_\omega \cup (1-x)_\omega)_\Sfr = {n \choose 2}\sigma^\omega(x)\sigma^\omega(1-x)$ and
\item $(({n \choose 2}\sigma \otimes \sigma,\sigma),x_\omega \cup \omega_\omega +\beta x_\omega)_\Sfr = {n \choose 2}\sigma^\omega(x)\sigma^\omega(\omega) + \sigma^\omega(x)$.
\end{enumerate}
The final ingredient is the definition of $\Mcal_K^*$, which sends a function $f : R(K) \rightarrow \Lambda$ to the function $x \mapsto (f(x_\omega \cup (1-x)_\omega),f(x_\omega \cup \omega_\omega+\beta x_\omega))$.
Now it's a simple matter of putting everything together.
Namely, for $\sigma,\tau \in \gfrak_1(K)$, $x \in \Kbb$, and $\Omega_K$ as defined above, one has:
\begin{enumerate}
\item $\Omega_K([\sigma,\tau])(x) = (\sigma^\omega(x)\tau^\omega(1-x)-\sigma^\omega(1-x)\tau^\omega(x), \ \sigma^\omega(x)\tau^\omega(\omega)-\sigma^\omega(\omega)\tau^\omega(x))$.
\item $\Omega_K(\sigma^\pi)(x) = ({n \choose 2} \cdot \sigma^\omega(x)\sigma^\omega(1-x), \ {n \choose 2}\sigma^\omega(x)\sigma^\omega(\omega)+\sigma^\omega(x))$.
\end{enumerate}
so that $\Omega_K([\sigma,\tau]) = \Phi^\omega(\sigma^\omega,\tau^\omega)$ and $\Omega_K(\sigma^\pi) = \Psi^\omega(\sigma^\omega)$.
Since $\gfrak_2(K)$ is (topologically) generated by elements of the form $[\sigma,\tau]$ and $\sigma^\pi$ as $\sigma,\tau \in \gfrak_1(K)$ vary, $\Ffrak_K^\omega$ is generated by $\Phi^\omega(f,g)$ and $\Psi^\omega(f)$ as $f,g \in \Hom(K^\times,\Lambda)$ vary, and $(\bullet)^\omega : \gfrak_1(K) \rightarrow \Hom(K^\times,\Lambda)$  is an isomorphism, we see that $\Ffrak_K^\omega$ is indeed the image of $\Omega_K$.
Thus, the calculation above completes the proof that $\Omega_K$ is an isomorphism between $\gfrak_2(K)$ and $\Ffrak_K^\omega$ which satisfies the requirements of Theorem \ref{thm:main-thm-ffrak}.

\subsection{The Proof of Theorem \ref{thm:main-thm-ffrak} -- Functoriality}
\label{sec:proof-theor-functoriality}
As noted above, $\gfrak_2(K)$ is topologically generated by elements of the form $[\sigma,\tau]$ and $\sigma^\pi$, as $\sigma,\tau \in \gfrak_1(K)$ vary, while $\Ffrak_K^\omega$ is (defined to be) topologically generated by $\Phi^\omega(f,g)$ and $\Psi^\omega(f)$, as $f,g \in \Hom(K^\times,\Lambda)$ vary.
Also recall that the image of $\Omega_K$ is precisely $\Ffrak_K^\omega$, so that $\Omega_K : \gfrak_2(K) \rightarrow \Ffrak_K^\omega$ is our desired isomorphism.
The functoriality in Theorem \ref{thm:main-thm-ffrak} now follows immediately from the fact that, for all fields $K$ as in the theorem, $\Omega_K$ is an isomorphism which sends $[\sigma,\tau]$ to $\Phi^\omega(\sigma^\omega,\tau^\omega)$ and $\sigma^\pi$ to $\Psi^\omega(\sigma^\omega)$.

More precisely, let us consider a situation where $K \hookrightarrow L$ is a field extension.
Then one obtains a canonical map $\phi : \Gal(L) \rightarrow \Gal(K)$ which induces canonical maps $\phi_i : \gfrak(L)_i \rightarrow \gfrak_i(K)$ for $i = 1,2$, which are compatible with $[\bullet,\bullet]$ and $(\bullet)^\pi$.
I.e. if $\sigma,\tau \in \gfrak_1(K)$, then $\phi_2([\sigma,\tau]) = [\phi_1(\sigma),\phi_1(\tau)]$ and $\phi_2(\sigma^\pi) = (\phi_1(\sigma))^\pi$.
Furthermore, one has $\res_K(\Phi^\omega(f,g)) = \Phi^\omega(f|_{K^\times},g|_{K^\times})$ and $\res_K(\Psi^\omega(f)) = \Psi^\omega(f|_{K^\times})$.
Lastly, $(\phi_1\sigma)^\omega = (\sigma^\omega)|_{K^\times}$.
Combining all these compatibility properties, along with the fact that $\Omega_K$ and $\Omega_L$ as both isomorphisms, we obtain the desired functoriality.

We now give an alternative proof of functoriality which uses the discussion of Remark \ref{remark:functoriality-general}.
First, observe that the canonical homomorphism $\phi_1^* : \H^2(\gfrak_1(K)) \rightarrow \H^2(\gfrak_1(L))$ restricts to a map $R(K) \rightarrow R(L)$, where $R(K)$ and $R(L)$ are as defined in \S\ref{sec:proof-theor-refthm-omegaK}.
Thus, as in Remark \ref{remark:functoriality-general}, we obtain a homomorphism $\phi_1^{**} : \Sfr_{R(L)} \rightarrow \Sfr_{R(K)}$ which is compatible with $\phi_2 : \gfrak_2(L) \rightarrow \gfrak_2(K)$ in the sense that the following diagram commutes:
\[
\xymatrix{
\gfrak_2(L) \ar[d]_{\phi_2} \ar[r]^{\Omega_{R(L)}} & \Sfr_{R(L)} \ar[d]^{\phi_1^{**}} \\
\gfrak_2(K)  \ar[r]_{\Omega_{R(K)}} & \Sfr_{R(K)}
}
\]
and the following compatibility properties hold:
\begin{enumerate}
\item $\phi_1^{**}(\Omega_{R(L)}([\sigma,\tau])) = \Omega_{R(K)}(\phi_2[\sigma,\tau]) = \Omega_{R(K)}([\phi_1\sigma,\phi_1\tau])$.
\item $\phi_1^{**}(\Omega_{R(L)}(\sigma^\pi)) = \Omega_{R(K)}(\phi_2(\sigma^\pi)) = \Omega_{R(K)}((\phi_1\sigma)^\pi)$.
\end{enumerate}

If we define $\Lbb := L \smallsetminus\{0,1\}$, then the map $R(K) \rightarrow R(L)$ is compatible with the embedding $\Kbb \hookrightarrow \Lbb$ in the sense that the following diagram commutes:
\[ 
\xymatrix{
\Kbb \sqcup \Kbb \ar@{^{(}->}[r]^{K \subset L} \ar[d]_{\Mcal_K} &\Lbb \sqcup \Lbb \ar[d]^{\Mcal_L}\\
R(K) \ar[r]_{\phi_1^*} & R(L)
}
\]
Thus, we obtain the following commutative diagram:
\[
\xymatrix{
\gfrak_2(L) \ar[d]_{\phi_2} \ar[r]^{\Omega_{R(L)}} & \Sfr_{R(L)} \ar[d]_{\phi_1^{**}} \ar@{^{(}->}[r]^-{\rm incl.} & \Fun(R(L),\Lambda) \ar@{^{(}->}[r]^-{\Mcal_L^*} \ar[d]^{\phi_1^{**}} & \Fun(\Lbb,\Lambda^2) \ar[d]^{(K \subset L)^*}\\
\gfrak_2(K)  \ar[r]_{\Omega_{R(K)}} & \Sfr_{R(K)} \ar@{^{(}->}[r]^-{\rm incl.} & \Fun(R(K),\Lambda) \ar@{^{(}->}[r]^-{\Mcal_K^*} & \Fun(\Kbb,\Lambda^2)
}
\]
and functoriality follows since $\Omega_L$ resp. $\Omega_K$ is the composition of the maps on the top resp. bottom row of the diagram above.
This completes the proof of Theorem \ref{thm:main-thm-ffrak}.

\section{Relations in Abelian-by-Central Galois Groups}
\label{sec:applications}
Theorem \ref{thm:main-thm-ffrak} can be used to determine every relation that occurs among the elements $[\sigma,\tau]$ and $\sigma^\pi$ within $\gfrak_2(K)$.
Below is an example of a result which restricts the types of relations that occur in these Galois groups. 
The following theorem proves, among other things, that every relation in $\gfrak_2(K)/(2\cdot \gfrak_2(K)^\pi)$, between elements of the form $[\sigma,\tau]$, can be determined by looking only at homomorphisms from $\Gal(K)/(\Gal(K)^{(3)} \cdot \Gal(K)^{2n})$ to $\Hcal_\Lambda$, the Heisenberg group over $\Lambda$.

First, we briefly recall some facts concerning convergence of elements in $\gfrak_1(K)$ and $\gfrak_2(K$).
Let $K$ be a field whose characteristic is prime to $n$ with $\mu_n \subset K$, and choose $\omega \in \mu_n$ a primitive $n$th root of unity.
Let $\sigma_i,\tau_i,\gamma_j \in \gfrak_1(K)$ be given, with $i,j$ varying over possibly infinite indexing sets.
Assume that the collection $([\sigma_i,\tau_i],\gamma_j^\pi)_{i,j}$ converges to $0$ in $\gfrak_2(K)$.
Then the following three sums are well-defined elements of $\gfrak_2(K)$:
\begin{enumerate}
\item  $\sum_i[\sigma_i,\tau_i] + \sum_j\gamma_j^\pi$.
\item $\sum_i[\sigma_i,\tau_i]$.
\item $\sum_j\gamma_j^\pi$.
\end{enumerate}
Moreover, recall that, if $(\sigma_i,\tau_i,\gamma_j)_{i,j}$ converges to $0$ in $\gfrak_1(K)$, then the collection $([\sigma_i,\tau_i],\gamma_j^\pi)_{i,j}$ converges to $0$ in $\gfrak_2(K)$.
The connection with Kummer theory is as follows: a collection $(\sigma_i)_i$, of elements $\sigma_i \in \gfrak_1(K)$, converges to $0$ if and only if, for all $x \in K^\times$, all but finitely many of the $\sigma_i^\omega(x)$ vanish.

Recall that $\Fun(K\smallsetminus\{0,1\},\Lambda^2)$, endowed with the compact-open topology, is naturally an abelian profinite group.
In this case, a collection $(f_i)_i$ of elements $f_i \in \Fun(K\smallsetminus\{0,1\},\Lambda^2)$ converges to $0$ if and only if, for all $x \in K \smallsetminus\{0,1\}$, all but finitely many of the $f_i(x)$ vanish.
In this case, the sum $\sum_i f_i$ yields a well-defined element of $\Fun(K\smallsetminus\{0,1\},\Lambda^2)$ which is the function defined by $x \mapsto \sum_i f_i(x)$; since all but finitely many of the $f_i(x)$ are $0$, the sum $\sum_i f_i(x)$ is well-defined element in $\Lambda^2$.

Lastly, recall that the isomorphism $\Omega_K : \gfrak_2(K) \rightarrow \Ffrak_K^\omega$ from Theorem \ref{thm:main-thm-ffrak} is an isomorphism of abelian profinite groups.
In particular, $\Omega_K$ identifies $\gfrak_2(K)$ with a closed subgroup of $\Fun(K\smallsetminus\{0,1\},\Lambda^2)$.
Thus, we see that the collection $([\sigma_i,\tau_i],\gamma_j^\pi)_{i,j}$ converges to $0$ in $\gfrak_2(K)$ if and only if 
\[ (\Omega_K([\sigma_i,\tau_i]),\Omega_K(\gamma_j^\pi))_{i,j} = (\Phi^\omega(\sigma_i^\omega,\tau_i^\omega),\Psi^\omega(\gamma_j^\omega))_{i,j}\]
converges to $0$ in $\Fun(K\smallsetminus\{0,1\},\Lambda^2)$.

In Theorem \ref{thm:application-to-relations} below, the various statements deal with possibly infinite sums.
However, one of the assumptions of the theorem, that $(\sigma_i,\tau_i)_i$ converges to $0$ in $\gfrak_1(K)$, ensures that all of these possibly infinite sums are actually well-defined elements in their respective profinite groups, as discussed above.
Because we use it in the statement of Theorem \ref{thm:application-to-relations}, we also recall the well-known fact that the map $2 \cdot (\bullet)^\pi : \gfrak_1 \rightarrow \gfrak_2$ is $\Lambda$-linear; this fact also follows easily from Theorem \ref{thm:main-general-groups}.

\begin{theorem}
\label{thm:application-to-relations}
Let $K$ be a field whose characteristic is prime to $n$ with $\mu_{2n} \subset K$.
Choose a primitive $n$th root of unity $\omega \in \mu_n$.
Let $\sigma_i,\tau_i \in \gfrak_1(K)$ be given with the indexing set for $i$ possibly infinite, and assume that $(\sigma_i,\tau_i)_i$ converges to $0$ in $\gfrak_1(K)$.
Then the following are equivalent:
\begin{enumerate}
\item $\sum_i(\sigma_i^\omega(x) \cdot \tau_i^\omega(1-x) - \sigma_i^\omega(1-x) \cdot\tau_i^\omega(x)) = 0$ for all $x \in K \smallsetminus\{0,1\}$.
\item $\sum_i [\sigma_i,\tau_i] = \sum_i (b_i \cdot (2 \cdot \sigma_i^\pi) - a_i \cdot(2 \cdot \tau_i^\pi))$ where $\sigma_i^\omega(\omega) = 2\cdot a_i$ and $\tau_i^\omega(\omega) = 2\cdot b_i$ (such $a_i,b_i \in \Lambda$ exist since $\omega \in K^{\times 2}$).
\item $\sum_i [\sigma_i,\tau_i] \in \langle 2 \cdot \sigma_i^\pi,2 \cdot \tau_i^\pi \rangle_i$.
\item $\sum_i [\sigma_i,\tau_i] \in 2 \cdot \gfrak_1(K)^\pi$.
\item For all homomorphisms $\phi : \Gal(K)/\Gal(K)^{(3)} \rightarrow \Hcal_\Lambda$, one has $\sum_i[\phi_1\sigma_i,\phi_1\tau_i] = 0$, as elements of $\gfrak_2(\Hcal_\Lambda)$.
\item $\sum_i(\sigma_i^\omega(x) \cdot \tau_i^\omega(y) - \sigma_i^\omega(y) \cdot\tau_i^\omega(x)) = 0$, for all $x,y \in K^\times$ such that the image of $x_\omega \cup y_\omega$ vanishes in $\H^2(\Gal(K))$.
\item For all $x \in K \smallsetminus\{0,1\}$, there exists some solution $\phi : \Gal(K)/\Gal(K)^{(3)} \rightarrow \Hcal_\Lambda$ to the embedding problem $\mathcal{E}_{x,1-x}$, with $\sum_i[\phi_1\sigma_i,\phi_1\tau_i] = 0$, as elements of $\gfrak_2(\Hcal_\Lambda)$.
\end{enumerate}
\end{theorem}
\begin{proof}
To prove the equivalence of these statements, we will show the following implications:
$(1) \Rightarrow (2) \Rightarrow (3) \Rightarrow (4) \Rightarrow (1)$, $(4) \Rightarrow (5) \Rightarrow (6) \Rightarrow (1)$ and $(5) \Rightarrow (7) \Rightarrow (1)$.

Let us assume (1).
Then by Theorem \ref{thm:main-thm-ffrak}, 
\[ \Omega_K(\sum_i[\sigma_i,\tau_i])(x) = (0,\sum_i(\sigma_i^\omega(x)\tau_i^\omega(\omega)-\sigma_i^\omega(\omega)\tau_i^\omega(x))). \]
On the other hand, recall that $\Omega_K(2 \cdot \sigma^\pi)(x) = (0,2 \cdot \sigma^\omega(x))$.
Since $\omega$ is a square in $K$, we obtain (2) as follows.
Choose $a_i,b_i \in \Lambda$ such that $\sigma_i^\omega(\omega) = 2 \cdot a_i$ and $\tau_i^\omega(\omega) = 2 \cdot b_i$.
Then the equation above implies that:
\[ \sum_i [\sigma_i,\tau_i] = \sum_i( b_i \cdot (2\cdot \sigma_i^\pi) - a_i \cdot (2\cdot \tau_i^\pi)). \]
The implications $(2) \Rightarrow (3) \Rightarrow (4)$ are trivial.

Assume (4).
Then, by Theorem \ref{thm:main-thm-ffrak}, $\Omega_K(\sum_i [\sigma_i,\tau_i]) = (0,2\cdot\gamma_\omega)$ for some $\gamma \in \gfrak_1(K)$.
But then (1) follows immediately from Theorem \ref{thm:main-thm-ffrak} since the first component of $\Omega_K(\sum_i[\sigma_i,\tau_i])(x)$ is precisely 
\[ \sum_i(\sigma_i^\omega(x) \cdot \tau_i^\omega(1-x) - \sigma_i^\omega(1-x) \cdot\tau_i^\omega(x)). \]
Thus (1), (2), (3) and (4) are equivalent.

Assume (4) and let $\phi : \Gal(K)/\Gal(K)^{(3)} \rightarrow \Hcal_\Lambda$ be given as in (5).
Recall that for all $\gamma \in \gfrak_1(\Hcal_\Lambda)$, the element $2 \cdot \gamma^\pi \in \gfrak_2(\Hcal_\Lambda)$ is trivial.
Therefore, $\phi_2(\sum_i [\sigma_i,\tau_i])$ is trivial.
Thus we obtain (5).

The implication $(5) \Rightarrow (6)$ follows immediately from Proposition \ref{prop:heisenberg}, recalling that for $\sigma \in \gfrak_1(K)$ and $x \in K^\times$ one has $\sigma^\omega(x) = \sigma(x_\omega)$. The implication $(6) \Rightarrow (1)$ is obvious since, for all $x \in K\smallsetminus\{0,1\}$, the image of $x_\omega \cup (1-x)_\omega$ is trivial in $\H^2(\Gal(K))$.
Lastly, $(5) \Rightarrow (7)$ is trivial, while $(7) \Rightarrow (1)$ is, again, simply applying Proposition \ref{prop:heisenberg}.
\end{proof}

\subsection{Weakly-Commuting-Liftable Pairs}
\label{sec:aappl-theor-comm}

Let $K$ be a field whose characteristic is relatively prime to $n$ with $\mu_{2n} \subset K$.
We recall the definition of a CL-pair from \cite{Topaz2012c}: a pair of elements $\sigma,\tau \in \gfrak_1(K)$ is called a {\bf commuting-liftable} pair (or a {\bf CL}-pair for short) provided that $[\sigma,\tau] \in \langle 2 \cdot \sigma^\pi, 2 \cdot \tau^\pi \rangle$.
More generally, we will say that $\sigma,\tau$ are a {\bf weakly-commuting-liftable} pair (or {\bf WCL}-pair for short) provided that $[\sigma,\tau] \in 2 \cdot \gfrak_2(K)^\pi$.
Clearly, any CL-pair is a WCL-pair, and thus the condition defining a WCL-pair is \emph{a priori} weaker than the condition defining a CL-pair.
The equivalence of (3) and (4) in Theorem \ref{thm:application-to-relations} immediately implies the \emph{equivalence} of the two notions.
\begin{corollary}
\label{cor:cl-pair}
Let $K$ be a field whose characteristic is relatively prime to $n$ with $\mu_{2n} \subset K$.
Let $\sigma,\tau \in \gfrak_1(K)$ be given.
Then the following are equivalent:
\begin{enumerate}
\item $\sigma,\tau$ form a CL-pair.
\item $\sigma,\tau$ form a WCL-pair.
\end{enumerate}
\end{corollary}

\begin{remark}
\label{remark:cl-vs-wcl}
Let $\ell$ be a prime and let $m \geq 1$ be given.
For a field $K$ consider $\Gc_K := \Gal(K(\ell)|K)$, the maximal pro-$\ell$ Galois group of $K$.
Then, for any $M \gg m$, the main theorems in \cite{Topaz2012c} show how to detect minimized inertia and decomposition subgroups of $\Gc_K/([\Gc_K,\Gc_K] \cdot \Gc_K^{\ell^m})$, using $\Gc_K/([\Gc_K,\Gc_K] \cdot \Gc_K^{\ell^M})$ endowed with the \emph{subset of CL-pairs} in $(\Gc_K/([\Gc_K,\Gc_K] \cdot \Gc_K^{\ell^M}))^2$, as long as $\Char K \neq \ell$ and $\mu_{2 \ell^M} \subset K$.
Moreover, if $m = 1$ then $M = 1$ suffices, and loc.cit. computes an explicit $M$ which works in general, depending on $\ell$ and $m$.

Using Corollary \ref{cor:cl-pair}, we obtain an immediate strengthening of the main results of \cite{Topaz2012c}, by using WCL-pairs instead of CL-pairs.
Namely, one can now detect minimized inertia and decomposition subgroups of $\Gc_K/([\Gc_K,\Gc_K] \cdot \Gc_K^{\ell^m})$, using $\Gc_K/([\Gc_K,\Gc_K] \cdot \Gc_K^{\ell^M})$ endowed with the \emph{subset of WCL-pairs} in $(\Gc_K/([\Gc_K,\Gc_K] \cdot \Gc_K^{\ell^M}))^2$, with $K$ and $M$ as above, since, by Corollary \ref{cor:cl-pair}, this is exactly the set of CL-pairs.
\end{remark}

\appendix
\section{Cocycle calculations}
\label{sec:cocycle-calculations}

The goal for this section will be to provide a complete and self-contained proof of Theorem \ref{thm:cocycle-calculations}.
When dealing with finitely generated pro-$p$ groups, the cocycle calculations that follow are fairly well-known and were carried out, using minimal free pro-$p$ presentations and the Burnside basis theorem, by {\sc Labute} \cite{Labute1967}; see also the exposition in \cite{Neukirch2008} Propositions 3.9.13 and 3.9.14.
We generalize these calculations below by completely avoiding the use of minimal free presentations (which need not exist in general).

In this section we will use the notation of \S\ref{sec:cohom-frat-covers}.
Namely, $\Gfrak$ is an arbitrary profinite group such that $\gfrak_1(\Gfrak) \cong (\Z/n)^I$ for some indexing set $I$, $\gfrak_i := \gfrak_i(\Gfrak)$, and $\Kfr := \H^1(\gfrak_1,\Lambda) = \H^1(\Gfrak,\Lambda)$.
We also canonically identify $\gfrak_1$ with $\Hom(\Kfr,\Lambda)$ via the perfect pairing $\gfrak_1 \times \Kfr \rightarrow \Lambda$.

For $x \in \Kfr$ and $\sigma \in \gfrak_1$, we denote by $f_x(\sigma)$ the unique integer $0 \leq f_x(\sigma) < n$ so that $f_x(\sigma) \mod n = \sigma(x)$.
For $x,y \in \Kfr$ and $\sigma,\tau \in \gfrak_1$, we define:
\[ {\rm U}_{x,y}(\sigma,\tau) := \sigma(x)\cdot\tau(y) \]
and 
\[ {\rm B}_x(\sigma,\tau) := \begin{cases} 0, \ \text{ if } \  f_x(\sigma) + f_x(\tau) < n \\ 1,  \ \text{ if } \  f_x(\sigma) + f_x(\tau) \geq n. \end{cases} \]

It is immediately clear that ${\rm U}_{x,y} : \gfrak_1^2 \rightarrow \Lambda$ is a $2$-cocycle which represents the class of $x \cup y$ in $\H^2(\gfrak_1)$. 
On the other hand, a simple calculation using the definition of the Bockstein map as the connecting homomorphism $\H^1(\gfrak_1) \rightarrow \H^2(\gfrak_1)$ associated to
\[ 1 \rightarrow \Z/n \xrightarrow{n} \Z/n^2 \rightarrow \Z/n \rightarrow 1, \]
shows that ${\rm B}_x : \gfrak_1^2 \rightarrow \Lambda$ is a $2$-cocycle which represents the class of $\beta x \in \H^2(\gfrak_1)$.

Before we proceed, let us make a couple of observations.
First, if $\sigma(x) = \sigma'(x)$ and $\tau(y) = \tau'(y)$ then ${\rm U}_{x,y}(\sigma,\tau) = {\rm U}_{x,y}(\sigma',\tau')$.
Similarly, if $\sigma(x) = \sigma'(x)$ and $\tau(x) = \tau'(x)$, then ${\rm B}_x(\sigma,\tau) = {\rm B}_x(\sigma',\tau')$.
In the remainder of this section, we will write $\gfrak_1$ multiplicatively (contrary to our previous convention), in order to avoid confusion with multiplicative notation involving elements of $\Gfrak$, which we will need below in the proof of Theorem \ref{thm:cocycle-calculations}.

\begin{proposition}
\label{prop:cocycles}
Let $x,y,z \in \Kfr$ and $\sigma,\tau \in \gfrak_1$ be given. 
Then the following identities hold:
\begin{enumerate}
\item ${\rm U}_{x,y}(\sigma,\tau)+{\rm U}_{x,y}(\tau^{-1},\sigma\tau) - {\rm U}_{x,y}(\tau^{-1},\tau) = \sigma(x)\tau(y)-\sigma(y)\tau(x)$.
\item $\sum_{i = 0}^{n-1} {\rm U}_{x,y}(\sigma^i,\sigma) = {n \choose 2} \sigma(x)\sigma(y)$.
\item ${\rm B}_x(\sigma,\tau)+{\rm B}_x(\tau^{-1},\sigma\tau)-{\rm B}_x(\tau^{-1},\tau) = 0$.
\item $\sum_{i= 0}^{n-1} {\rm B}_x(\sigma^i,\sigma) = \sigma(x)$.
\end{enumerate}
\end{proposition}
\begin{proof}
To 1. One has
\begin{align*}
{\rm U}_{x,y}(\sigma,\tau)+{\rm U}_{x,y}(\tau^{-1},\sigma\tau) - {\rm U}_{x,y}(\tau^{-1},\tau) &= \sigma(x)\tau(y)-\tau(x)\cdot (\sigma(y)+\tau(y)) + \tau(x)\tau(y) \\
&= \sigma(x)\tau(y)-\sigma(y)\tau(x).
\end{align*}

To 2. One has
\begin{align*}
\sum_{i = 0}^{n-1} {\rm U}_{x,y}(\sigma^i,\sigma) &= \sum_{i = 0}^{n-1} i \cdot \sigma(x) \sigma(y) \\
&= {n \choose 2} \cdot \sigma(x) \sigma(y).
\end{align*}

To 3. If $\tau(x) = 0$, it follows immediately from the definition that the expression ${\rm B}_x(\sigma,\tau)+{\rm B}_x(\tau^{-1},\sigma\tau)-{\rm B}_x(\tau^{-1},\tau)$ vanishes (all three terms in the expression are $0$ in this case).
Let us therefore assume that $\tau(x) \neq 0$ and thus $f_x(\tau) \neq 0$.
In this case, one has $f_x(\tau^{-1}) = n-f_x(\tau)$ and so ${\rm B}_x(\tau^{-1},\tau) = 1$.
If $f_x(\sigma)+f_x(\tau) < n$, then ${\rm B}_x(\sigma,\tau) = 0$ while ${\rm B}_x(\tau^{-1},\sigma\tau) = 1$ since
\begin{align*}
f_x(\tau^{-1})+f_x(\sigma\tau) &= (n-f_x(\tau))+(f_x(\sigma)+f_x(\tau)) \\
&= n+f_x(\sigma) \geq n.
\end{align*}
If, on the other hand, $f_x(\sigma)+f_x(\tau) \geq n$, then ${\rm B}_x(\sigma,\tau) = 1$ while ${\rm B}_x(\tau^{-1},\sigma\tau) = 0$ since
\begin{align*}
f_x(\tau^{-1})+f_x(\sigma\tau) &= (n-f_x(\tau))+(f_x(\sigma)+f_x(\tau)-n) \\
&= f_x(\sigma) < n.
\end{align*}
In either case, we see that the expression ${\rm B}_x(\sigma,\tau)+{\rm B}_x(\tau^{-1},\sigma\tau)-{\rm B}_x(\tau^{-1},\tau)$ vanishes, as required.

To 4. 
For $a',b' \in \Z_{\geq 0}$, define 
\[ {\rm B}(a',b') := \begin{cases} 0,  \ \text{ if } \  a'+b' < n \\ 1,  \ \text{ if } \  a'+b' \geq n \end{cases} \]
so that ${\rm B}(f_x(\sigma),f_x(\tau)) = {\rm B}_x(\sigma,\tau)$.

Define $a := f_x(\sigma)$.
Let $g = {\rm gcd}(a,n)$ and denote by $e \in \Z$ the integer such that $a = g \cdot e$.
Observe that as $i$ varies from $0$ to $n-1$, the integer $f_x(\sigma^i)$ varies over $0,g,2g,\ldots,(\frac{n}{g}-1) \cdot g$ and each one of these integers occurs precisely $g$ times.
We calculate:
\begin{align*}
\sum_{i = 0}^{n-1} {\rm B}_x(\sigma^i,\sigma) &= \sum_{i = 0}^{n-1} {\rm B}(f_x(\sigma^i),f_x(\sigma)) \\
&= g \cdot \sum_{j=0}^{\frac{n}{g}-1} {\rm B}(j \cdot g, e \cdot g).
\end{align*}
Note that the $j$ with $0 \leq j < \frac{n}{g}$ and $j \cdot g + e \cdot g \geq n$ (equivalently, ${\rm B}(j \cdot g, e \cdot g) = 1$) are precisely the $j$ such that $\frac{n}{g}-e \leq j < \frac{n}{g}$; there are precisely $e$ such integers $j$.
The other $j$ with $0 \leq j < \frac{n}{g}$ (i.e. those $j$ with $j \cdot g + e \cdot g < n$) satisfy ${\rm B}(j \cdot g,e \cdot g) = 0$.
Thus 
\[ \sum_{j=0}^{\frac{n}{g}-1} {\rm B}(j \cdot g, e \cdot g) = e \]
and therefore 
\[ \sum_{i = 0}^{n-1} {\rm B}_x(\sigma^i,\sigma) = g \cdot e = a = \sigma(x). \]
This completes the proof of the proposition.
\end{proof}

We now give a proof of Theorem \ref{thm:cocycle-calculations}.
\begin{proof}[Proof of Theorem \ref{thm:cocycle-calculations}]
Let $\tsigma$ resp. $\ttau \in \Gfrak$ be arbitrary elements, and denote their images in $\gfrak_1$ by $\sigma$ resp. $\tau$.
To simplify the notation, for $\tsigma,\ttau$ as above, we will define ${\rm U}_{x,y}(\tsigma,\ttau) := {\rm U}_{x,y}(\sigma,\tau)$ and ${\rm B}_z(\tsigma,\ttau) := {\rm B}_z(\sigma,\tau)$.

Let $\eta$ be an element of $\ker(\H^2(\gfrak_1) \rightarrow \H^2(\Gfrak))$.
Choose a representation $\eta = \sum_i x_i \cup y_i + \sum_j \beta z_j$ with $x_i,y_i,z_j \in \Kfr$, as in Lemma \ref{lemma-perf-pairing-H2}.
Consider the cocycle $\xi : \Gfrak^2 \rightarrow \Lambda$:
\[ \xi : (\tsigma,\ttau) \mapsto \sum_i {\rm U}_{x_i,y_i}(\tsigma,\ttau) + \sum_j {\rm B}_{z_j}(\tsigma,\ttau). \]
Then $\xi$ is a $2$-cocycle which represents the inflation of $\eta$ to $\Gfrak$.
Since the cohomology class of $\xi$ is trivial in $\H^2(\Gfrak)$, there exists $u : \Gfrak \rightarrow \Lambda$, a cochain with $du = \rho$.
In other words, 
\[ u(\tsigma\ttau) = u(\tsigma) + u(\ttau) - \xi(\tsigma,\ttau). \]
Furthermore, by adding a constant to $u$, we may assume without loss that $u(1) = 0$.
Thus:
\[ u(\tsigma^{-1}) = -u(\tsigma) + \xi(\tsigma^{-1},\tsigma). \]
The restriction of $u$ to $\Gfrak^{(2)}$ is the unique element $u \in \H^1(\Gfrak^{(2)})^{\gfrak_1}$ with $d_2u = \eta$ (see e.g. \cite{Neukirch2008} Propositions 1.6.6 and 2.4.3).
Thus, it suffices to calculate $u(\tsigma^{-1}\ttau^{-1}\tsigma\ttau) = ([\sigma,\tau],u)_{\gfrak_2}$ and $u(\tsigma^n) = (\sigma^\pi,u)_{\gfrak_2}$.

First we calculate $u(\tsigma^n)$:
\begin{align*}
u(\tsigma^n) &= u(\tsigma^{n-1}) + u(\tsigma) - \xi(\tsigma^{n-1},\tsigma) \\
 &= u(\tsigma^{n-2}) + 2\cdot u(\tsigma) - (\xi(\tsigma^{n-2},\tsigma)+\xi(\tsigma^{n-1},\tsigma)) \\
 &= \cdots = n \cdot u(\tsigma) - \sum_{l = 1}^{n-1} \xi(\tsigma^l,\tsigma) \\
 &= -\sum_{l = 1}^{n-1}\xi(\tsigma^l,\tsigma).
\end{align*}
Since $\xi(1,\tsigma) = 0$, we get:
\[ u(\tsigma^n) = -\sum_{l = 0}^{n-1}\xi(\tsigma^l,\tsigma). \]
By Proposition \ref{prop:cocycles}, we see that 
\[ u(\tsigma^n) = -\left( {n \choose 2} \cdot \sum_i\sigma(x_i)\sigma(y_i) + \sum_j \sigma(z_j) \right). \]

Now we calculate $u(\tsigma^{-1}\ttau^{-1}\tsigma\ttau)$:
\begin{align*}
u(\tsigma^{-1}\ttau^{-1}\tsigma\ttau) &= u(\tsigma^{-1})+u(\ttau^{-1}\tsigma\ttau) - \xi(\tsigma^{-1},\ttau^{-1}\tsigma\ttau) \\
&= u(\tsigma^{-1})+u(\ttau^{-1}\tsigma\ttau) - \xi(\tsigma^{-1},\tsigma) \\
&= -u(\tsigma)+u(\ttau^{-1}\tsigma\ttau) +\xi(\tsigma^{-1},\tsigma)-\xi(\tsigma^{-1},\tsigma) \\
&= -u(\tsigma)+u(\ttau^{-1})+u(\tsigma\ttau)-\xi(\ttau^{-1},\tsigma\ttau) \\
&= -u(\tsigma)-u(\ttau)+u(\tsigma\ttau)+\xi(\ttau^{-1},\ttau)-\xi(\ttau^{-1},\tsigma\ttau) \\
&= -u(\tsigma)-u(\ttau)+u(\tsigma)+u(\ttau)+\xi(\ttau^{-1},\ttau)-\xi(\ttau^{-1},\tsigma\ttau)-\xi(\tsigma,\ttau) \\
&= -(\xi(\tsigma,\ttau)+\xi(\ttau^{-1},\tsigma\ttau)-\xi(\ttau^{-1},\ttau)).
\end{align*}
Again by Proposition \ref{prop:cocycles}, we see that 
\[ u(\tsigma^{-1}\ttau^{-1}\tsigma\ttau) = -\left(\sum_i \sigma(x_i)\tau(y_i)-\sigma(y_i)\tau(x_i)\right). \]
This completes the proof of Theorem \ref{thm:cocycle-calculations}.
\end{proof}

\bibliography{../refs} 

\end{document}